\documentclass{article}

\usepackage{microtype}
\usepackage{graphicx}
\usepackage{subfigure}
\usepackage{booktabs} 

\usepackage{algorithm,algorithmic}
\usepackage{amsfonts}
\usepackage{amsthm}
\usepackage{amsmath}
\usepackage{amssymb}
\usepackage{enumitem}
\usepackage{bbm,verbatim}
\usepackage{multirow}
\usepackage{accents}
\usepackage{mathtools}
\usepackage{csquotes}

\newtheorem{theorem}{Theorem}[section]

\newtheorem{lemma}[theorem]{Lemma}

\newcommand{\norm}[1]{\lVert #1\rVert}

\usepackage{hyperref}

\usepackage[accepted]{icml2019}

\icmltitlerunning{Stochastic Variance-Reduced Heavy Ball Power Iteration}

\begin{document}

\twocolumn[
\icmltitle{Stochastic Variance-Reduced Heavy Ball Power Iteration}




\begin{icmlauthorlist}
\icmlauthor{Cheolmin Kim}{nu}
\icmlauthor{Diego Klabjan}{nu}
\end{icmlauthorlist}

\icmlaffiliation{nu}{Department of Industrial Engineering and Management Sciences, Northwestern University, Evanston, Illinois, USA}

\icmlcorrespondingauthor{Cheolmin Kim}{cheolminkim2019@u.northwestern.edu}

\icmlkeywords{Stochastic, PCA, Heavy Ball, Acceleration, Variance-reduced, Power iteration}

\vskip 0.3in
]



\printAffiliationsAndNotice{}  

\begin{abstract}
We present a stochastic variance-reduced heavy ball power iteration algorithm for solving PCA and provide a convergence analysis for it.
The algorithm is an extension of heavy ball power iteration, incorporating a step size so that progress can be controlled depending on the magnitude of the variance of stochastic gradients.
The algorithm works with any size of the mini-batch, and if the step size is appropriately chosen, it attains global linear convergence to the first eigenvector of the covariance matrix in expectation.
The global linear convergence result in expectation is analogous to those of stochastic variance-reduced gradient methods for convex optimization but due to non-convexity of PCA, it has never been shown for previous stochastic variants of power iteration since it requires very different techniques.
We provide the first such analysis and stress that our framework can be used to establish convergence of the previous stochastic algorithms for any initial vector and in expectation.
Experimental results show that the algorithm attains acceleration in a large batch regime, outperforming benchmark algorithms especially when the eigen-gap is small.
\end{abstract}

\section{Introduction}
Principal component analysis (PCA) is a fundamental tool for dimensionality reduction in machine learning and statistics. Given a data matrix $A=[a_1 a_2 \ldots a_n] \in \mathbb{R}^{d \times n}$ consisting of $n$ data vectors $a_1,a_2,\ldots,a_n$ in $\mathbb{R}^d$, PCA finds a direction $w$ onto which the projections of the data vectors have the largest variance. 
Assuming that the data vectors are standardized with a mean of zero and standard deviation of one, the PCA problem can be formulated as
\begin{equation} \label{prob:pca}
\begin{aligned} 
& \textrm{maximize} \quad f(w) = \frac1{2n}\sum_{i=1}^n (a_i^T w)^2 \\
& \textrm{subject to} \quad \norm{w}_2 = 1.
\end{aligned}
\end{equation}

Letting $C=\frac{1}{n} AA^T \in \mathbb{R}^{d \times d}$ be the covariance matrix, we can write the objective function as $f(w) = \frac1{2}w^TCw$. As the largest eigenvector $u_1$ of $C$ maximizes $f(w)$, one can solve \eqref{prob:pca} by computing the singular value decomposition (SVD) of $A$. However, the runtime of SVD is $\mathcal{O}(\text{min}\{nd^2,n^2d\})$, which can be prohibitive in a large-scale setting. 

An alternative way to solve \eqref{prob:pca} is to use power iteration \citep{golub2012matrix} which repeatedly applies the following update step at each iteration
\begin{align}
w_{t+1} = C w_{t}, \quad w_{t+1} = \frac{w_{t+1}}{\|w_{t+1}\|_2}. \label{eq:Power}
\end{align}
Since the gradient $\nabla f(w_t)$ is equal to $Cw_t$, the above update rule can be interpreted as obtaining the next iterate $w_{t+1}$ by normalizing the gradient of the current iterate $w_t$.
A sequence of iterates $\{w_t\}$ generated by power iteration \eqref{eq:Power} is guaranteed to achieve an $\epsilon$-accurate solution after $\mathcal{O}\big(\frac{1}{\Delta} \text{log} \frac{1}{\epsilon}\big)$ iterations, exhibiting linear convergence where $\Delta = \lambda_1 - \lambda_2$ is the eigen-gap and $\lambda_1 > \lambda_2 \geq \ldots \geq \lambda_d \geq 0$ are the eigenvalues of $C$.
As each iteration involves multiplying a vector $w_t$ with the matrix $C$, the total runtime becomes $\mathcal{O}\big(nd\frac{1}{\Delta} \text{log} \frac{1}{\epsilon}\big)$. 
If $n$ and $d$ are both large, the runtime of power iteration is better than that of SVD. 
Nonetheless, it still largely depends on $n$ and can be prohibitive when $\Delta$ is small. In order to reduce the dependence on $\Delta$ or $n$, the following variants of power iteration have been developed.

To reduce the dependence on $\Delta$, \cite{xu2018accelerated} proposed power iteration with momentum (Power+M), which is a simple variant of power iteration utilizing the momentum idea of \citep{polyak1964some}. With the additional momentum term, it can be written as
\begin{align}
w_{t+1} = 2 Cw_t - \beta w_{t-1}.
\label{eq:Power-momentum}
\end{align}
In heavy ball power iteration \eqref{eq:Power-momentum}, we have
\begin{align*}
(u_k^Tw_t)^2 \leq \beta^{t} (u_k^Tw_0)^2    
\end{align*}
if $\lambda_k \leq \beta$ and  
\begin{align*}
(u_k^Tw_t)^2 = \Theta \Big( \Big( \lambda_k + \sqrt{\lambda_k^2 - \beta} \Big)^{2t} \Big) (u_k^Tw_0)^2
\end{align*}
otherwise. 
Therefore, if $\beta = \lambda_2^2$, it achieves the optimal rate of convergence resulting in the runtime of $\mathcal{O}\big(nd\frac{1}{\sqrt{\Delta}} \text{log} \frac{1}{\epsilon}\big)$, which greatly improves the dependence on $\Delta$.


\begin{table*}[ht]
\caption{Comparison of stochastic variance-reduced methods for PCA and their convergence analyses. Algorithms are compared for two regimes (small-batch and large-batch), and types of convergence guarantee and conditions for the angle between an initial iterate $\tilde{w}_0$ and the first loading vector $u_1$ are summarized. \enquote{Local} means that there is a restriction on the angle while \enquote{global} implies no such restriction.}
\label{table:summary}
\vskip 0.15in
\begin{center}
\begin{small}
\begin{sc}
\begin{tabular}{l|cc|cc|c}
\toprule
Algorithm & 
\begin{tabular}[c]{@{}c@{}}\ Convergence  \\ (small-batch) \end{tabular}  & \begin{tabular}[c]{@{}c@{}}\ Acceleration  \\ (large-batch) \end{tabular} & 
\multicolumn{2}{c|}{\begin{tabular}[c]{@{}c@{}} Convergence Guarantee \end{tabular}} & Reference \\
\midrule
VR-PCA    	& $\surd$ 	& $\times$	& Probabilistic & Local & \cite{shamir2015stochastic} \\
VR Power+M  & $\times$	& $\surd$	& Probabilistic & Local & \cite{xu2018accelerated} \\
\textbf{VR HB Power} & $\surd$ 	& $\surd$	& Expectation & Global & \textbf{(This Paper)}       \\
\bottomrule
\end{tabular}
\end{sc}
\end{small}
\end{center}
\vskip -0.1in
\end{table*}

On the other hand, a stochastic algorithm utilizing a stochastic gradient $a_{i_t} a_{i_t}^T w_t$ rather than a full gradient $Cw_t$ is introduced in \citep{oja1982simplified}.
Since it requires just one data vector at a time, the computational cost per iteration is significantly reduced. However, due to the variance of stochastic gradients, a sequence of diminishing step sizes needs to be adopted in order to converge, making its progress slow near an optimum.
Built on a recent stochastic variance-reduced gradient (SVRG) technique \citep{johnson2013accelerating}, \cite{shamir2015stochastic} proposed a stochastic variance-reduced version of Oja's algorithm (VR-PCA).
By utilizing stochastic variance-reduced gradients with a constant step size, this sophisticated stochastic algorithm attains linear convergence, reducing the total runtime required to obtain an $\epsilon$-accurate solution to $\mathcal{O}(d(n+\frac{1}{\Delta^2}) \text{log} \frac{1}{\epsilon})$. 

Other works on the power method include the noisy \cite{hardt2014noisy}, coordinate-wise \cite{lei2016coordinate}, and shifted-and-inverted \cite{garber2015fast,garber2016faster} power methods. The noisy power method considers the power method in noise setting and \cite{balcan2016improved} extended it, providing an improved gap-dependency analysis. The shifted-and-inverted power method reduces the PCA problem to solving a series of convex least squares problems, which can be solved by optimization algorithms such as coordinate-descent \cite{wang2018efficient}, SVRG \cite{garber2015fast,garber2016faster}, accelerated gradient descent or accelerated SVRG \cite{allen2016lazysvd}, and Riemannian gradient descent \cite{xu2018gradient}. Due to the presence of fast least squares solvers, the shifted-and-inverted approach has received much attention. However, since it involves solving a series of optimization problems, it is not simple to implement and hard to parallelize while the variance-reduced power methods are easy to implement and a single iteration can be parallelized in the obvious way.

In this paper, we present a stochastic variance-reduced algorithm (VR HB Power) for heavy ball power iteration \eqref{eq:Power-momentum}. While a stochastic variance-reduced power iteration with momentum (VR Power+M) is introduced in \cite{xu2018accelerated}, it is not practical to use since it requires that the size of the mini-batch needs to be sufficiently large. In this work, we enhance the algorithm by adding a step size which turns out to have a big impact. By incorporating the step size, the proposed algorithm can work with any size of the mini-batch. Given that the step size is appropriately chosen depending on the size of the mini-batch, the algorithm attains linear convergence to the first loading vector as VR-PCA. Furthermore, if the size of the mini-batch is chosen to be large, it attains accelerated convergence, outperforming VR-PCA. Table \ref{table:summary} summarizes the state-of-the-art.

For the algorithm, we provide a novel convergence analysis where the resulting convergence statement provides a bound for the ratio of two expectations that goes to zero at a linear rate. This result is analogous to those of stochastic variance-reduced gradient methods for convex optimization and  stronger than probabilistic statements, appearing in \cite{shamir2015stochastic} and \cite{xu2018accelerated} in a sense that a probability parameter $\delta$ does not constraint the size of the mini-batch or the rate of convergence. Note that PCA studied herein is a non-convex problem and thus completely different techniques are needed. In order to obtain a convergence guarantee with high probability, the step size needs to be arbitrarily small in \cite{shamir2015stochastic} and the batch size needs to be arbitrarily large in \cite{xu2018accelerated}. Complementary to them, our analysis does not have such requirements and the convergence is deterministically guaranteed in terms of expectation terms.

Moreover, our analysis allows random initialization while VR-PCA and VR Power+M do not. Since random initialization of $\tilde{w}_0$ results in $|u_1^T \tilde{w}_0| \leq \mathcal{O}(1/\sqrt{d})$ with high probability \cite{shamir2015stochastic}, it is not trivial to obtain an initial iterate $\tilde{w}_0$ such that $|u_1^T \tilde{w}_0| \geq 1/2$, especially when $d$ is large. To handle this issue, an initialization scheme that samples a point from the standard Gaussian distribution in $\mathbb{R}^d$ and performs a single power iteration is presented in \cite{shamir2016fast}. After some stochastic iterations, this process yields an iterate $\tilde{w}_0$ satisfying $|u_1^T \tilde{w}_0| \geq 1/2$ with high probability. However, such an initialization scheme is not essentially necessary since VR-PCA practically works well with random initialization, as observed in \cite{shamir2015stochastic}. Our convergence analysis resolves this gap, showing that the rate of convergence does not depend on how far an iterate is from $u_1$ but is kept the same across iterations, as in the case of deterministic power iteration. The framework used in the convergence analysis is not specific to the presented algorithm; it can be extended to analyze other stochastic variance-reduced PCA algorithms such as VR-PCA or VR Power+M, deriving in expectation bounds for them and resolving their initialization issues.

Our work has the following contributions.
\begin{enumerate}
\item[1.] We present a stochastic variance-reduced algorithm for heavy ball power iteration, which works with any size of the mini-batch and attains acceleration in a large-batch regime. Since there is no constraint on the size of the mini-batch, it is more practical than VR Power+M, and it outperforms VR-PCA in a large batch-setting, especially when the eigen-gap $\Delta$ is small.

\item[2.] We provide a novel convergence analysis for the algorithm. The convergence result does not require a good initialization, yet provides a bound for the ratio of two expectation terms, which has a similar form to those of stochastic variance-reduced gradient methods for convex optimization. The framework for the convergence analysis is general, therefore can be used to analyze other stochastic variance-reduced PCA algorithms. To this end, we are the first to establish convergence of VR-PCA and VR Power+M for any initial vector and in expectation.

\item[3.] We report numerical experiments on diverse datasets to investigate the empirical performance of the algorithm. Experimental results show that our algorithm is more efficient than VR-PCA in a large-batch setting, especially when the eigen-gap is small and it outperforms VR-Power+M in all cases.
\end{enumerate}

The paper is organized as follows. We present the algorithm in Section~\ref{sec:algorithm} and the convergence analysis is provided in Section~\ref{sec:convergence-analysis}. Some practical considerations regarding the implementation of the algorithm are discussed in Section~\ref{sec:practical-considerations} and the experimental results are followed in Section~\ref{sec:numerical-experiments}.

\section{Algorithm}
\label{sec:algorithm}
In this section, we develop a stochastic variance-reduced algorithm for heavy ball power iteration \eqref{eq:Power-momentum}. For eigenvalues and eigenvectors of $C$, we assume that the eigenvalues $\lambda_1,
\lambda_2, \ldots, \lambda_d$ satisfy 
\begin{align*}
\lambda_1 > \lambda_2 \geq \ldots \geq \lambda_d \geq 0    
\end{align*}
and the eigenvectors $u_1, u_2, \ldots, u_d$ form an orthonormal basis. Since a symmetric matrix is orthogonally diagonalizable, we can assume such eigenvectors exist without loss of generality.

Variance reduction algorithms periodically compute the exact gradient which is then used in the inner stochastic gradient type updates. This exact gradient reduces the variance of this inner loop. In summary, a variance reduction algorithm has an outer loop and an inner loop.

Let $\tilde{w}_s$ and $w_t$ denote an outer-loop and inner-loop iterate, respectively. To get a stochastic variance-reduced gradient of an inner loop iterate $w_t$, we first decompose it into two parts as
\begin{align*}
w_t = \frac{(\tilde{w}_{s}^T w_{t})}{\|\tilde{w}_s\|^2} \tilde{w}_{s} + \bigg(I- \frac{\tilde{w}_{s} \tilde{w}_{s}^T}{\|\tilde{w}_s\|^2} \bigg) w_{t}
\end{align*}
using the outer loop iterate $\tilde{w}_{s}$.
In the above decomposition, the former term represents the projection of $w_{t}$ on $\tilde{w}_{s}$ while the latter term represents the remaining vector.
Utilizing the exact gradient $\tilde{g}$ at $\tilde{w}_{s}$, the exact gradient at the first term can be computed as 
\begin{align*}
\nabla f \bigg( \frac{(\tilde{w}_{s}^T w_{t})}{\|\tilde{w}_s\|^2} \tilde{w}_{s} \bigg) = \frac{(\tilde{w}_{s}^T w_{t})}{\|\tilde{w}_s\|^2} C\tilde{w}_{s} = \frac{(\tilde{w}_{s}^T w_{t})}{\|\tilde{w}_s\|^2} \tilde{g}.
\end{align*}
On the other hand, a stochastic sample $S$ is used to compute a stochastic gradient at the second term as
\begin{align*}
\frac{1}{|S|}\sum \limits_{l \in S} a_{l}a_{l}^T \bigg(I- \frac{\tilde{w}_{s} \tilde{w}_{s}^T}{\|\tilde{w}_s\|^2} \bigg) w_{t}
\end{align*}
resulting in the following stochastic variance-reduced gradient $g_t$ at $w_{t}$ as
\begin{align}
g_t = \frac{(\tilde{w}_{s}^T w_{t})}{\|\tilde{w}_s\|^2} \tilde{w}_{s} + \frac{1}{|S|}\sum \limits_{l \in S} a_{l}a_{l}^T \bigg(I- \frac{\tilde{w}_{s} \tilde{w}_{s}^T}{\|\tilde{w}_s\|^2} \bigg) w_{t}.
\label{eq:VR-gradient}
\end{align}
With the use of the stochastic variance-reduced gradient $g_t$, we obtain a stochastic variance-reduced heavy ball power iteration as 
\begin{align}
w_{t+1} \leftarrow 2 \big( (1-\eta)w_{t} + \eta g_t \big) - \beta w_{t-1}
\label{eq:VR-power-momentum}
\end{align}
where $\eta \in (0,1]$ is the step size and $\beta$ is the momentum parameter. 
Note that the deterministic update formula \eqref{eq:Power-momentum} can be obtained from \eqref{eq:VR-power-momentum} when the step size $\eta$ is set to 1 and the exact gradient $g_t = Cw_t$ is used.

As pointed out in \cite{JMLR:v18:16-410}, it is important to keep $g_t$ close to the exact gradient $\nabla f(w_t)$ to obtain stochastic acceleration. An obvious way to achieve this goal is to set the mini-batch size $|S|$ large. If the size of the mini-batch is large, the variance of the stochastic part in \eqref{eq:VR-gradient} becomes small, so that more accurate $g_t$ can be obtained. Another way to control the variance of $g_t$ is to decrease the step size $\eta$. 
If the step size $\eta$ is small, the angle between the outer iterate $\tilde{w}_{s}$ and the inner iterate $w_t$ can be kept close to $0$. If the outer iterate $\tilde{w}_{s}$ is closely aligned with the inner iterate $w_t$, the first term dominates the second term in \eqref{eq:VR-gradient} making $g_t$ close to the true gradient $\nabla f(w_t)$.

The mechanism of controlling the progress of the algorithm using the step size $\eta$ is not present in Power+M. As a result, it fails to converge unless the mini-batch size $|S|$ is sufficiently large. 
To the contrary, our algorithm works with any size of the mini-batch due to the presence of the step size $\eta$. If $|S|$ is small, since the variance of the stochastic part is large, a small $\eta$ needs to be chosen so that progress is made near the outer iterate $\tilde{w}_s$, making the inner iterate $w_t$ closely aligned with the outer iterate $\tilde{w}_s$. On the other hand, if $|S|$ is large, $g_t$ has a small variance even when it is far from the outer iterate $\tilde{w}_s$. Therefore, we can select a large $\eta$ to make rapid progress.

Summarizing all the above, we obtain VR HB Power exhibited in Algorithm~\ref{alg:VRPowerHB}.
\begin{algorithm}[h]
   \caption{VR HB Power}
   \label{alg:VRPowerHB}
\begin{algorithmic}
   \STATE {\bfseries Parameters:} step size $\eta$, momentum $\beta$, mini-batch size $|S|$, epoch length $m$ \\
   \STATE {\bfseries Input:} data vectors $a_i, i = 1,\ldots,n$ \\
   \STATE Randomly initialize an outer iterate $\tilde{w}_0$ \\
   \FOR{$s=0,1,\ldots$}
      \STATE $\tilde{g} \leftarrow \frac{1}{n}\sum_{l=1}^n a_l a_l^T {\tilde{w}_{s}}$ \\
      \STATE $w_0 \leftarrow \tilde{w}_{s}$
      \STATE $w_1 \leftarrow (1-\eta)w_0 + \eta\tilde{g}$ \\
      \FOR{$t = 1,2,\ldots,m-1$} 
	  \STATE Sample a mini-batch sample $S_{t}$ uniformly at random
      \STATE $g_t \leftarrow \frac{1}{|S_t|}\sum \limits_{l \in S_t} a_{l}a_{l}^T \Big( w_{t}- \frac{(w_t^T w_0) w_0}{\|w_0\|^2} \Big) + \frac{(w_t^T w_0)}{\|w_0\|^2} \tilde{g}$ \\
      \STATE $w_{t+1} \leftarrow 2 \big( (1-\eta)w_{t} + \eta g_t \big) - \beta w_{t-1}$ 
      \ENDFOR
      \STATE $\tilde{w}_s \leftarrow w_m$
   \ENDFOR
\end{algorithmic}
\end{algorithm}

\section{Convergence Analysis}
\label{sec:convergence-analysis}
In this section, we provide a convergence analysis for VR HB Power. Before presenting the convergence analysis, we first introduce some notations.
\begin{table*}[ht]
\caption{Comparison of the convergence rates of stochastic variance-reduced methods for PCA. The rates of convergence are obtained by analyzing the convergence of each algorithm using the framework presented in Section \ref{sec:convergence-analysis}. 
For each algorithm, $g$ and $h$ compose the convergence rate $\rho$ such that $\rho = g + h$.}
\label{table:summary-convergence-rate}
\vskip 0.15in
\begin{center}
\begin{small}
\begin{sc}
\begin{tabular}{l|c|c|r}
\toprule
Algorithm & Parameters & $g(\cdot)$ & \multicolumn{1}{c}{$h(\cdot)$} \\
\midrule
VR-PCA    & $\eta, K$  & 
\( \displaystyle 
\Bigg[
\frac{1 +\eta \lambda_2}{1 +\eta \lambda_1} \Bigg]^{2m}
\)
& $\eta^4 P_{1,m}(K)$ \\
VR Power+M  & $K$ & 
\(
\displaystyle 
\Bigg[
\frac{2 \lambda_2^m}
{\sum_{j=0}^1 ( \lambda_1 + (-1)^j \sqrt{\lambda_1 + \lambda_2} \sqrt{\Delta} )^{m}}
\Bigg]^2 \)
& $P_{1,m}(K)$ \\
VR HB Power & $\eta, K$ & \(
\displaystyle 
\Bigg[ \frac{2(1-\eta+\eta \lambda_2)^m}{ \sum_{j=0}^1 {\big( 1-\eta+\eta\lambda_1+(-1)^j \sqrt{2-2\eta+\eta(\lambda_1+\lambda_2)} \sqrt{\eta \Delta} \big)^m}} \Bigg]^2
\)
& $\eta^4 P_{1,m}(K)$ \\
\bottomrule
\end{tabular}
\end{sc}
\end{small}
\end{center}
\vskip -0.1in
\end{table*}

We define the sample covariance matrix $C_t$ at inner iteration $t$ and the projection matrix $P$ to the space orthogonal to the outer iterate $w_0=\tilde{w}_s$ as
\begin{align}
C_t = \frac{1}{|S_t|} \sum_{i_t \in S_t} a_{i_t} a_{i_t}^T, \quad P = I - \frac{w_0 w_0^T}{\|w_0\|^2}. \label{eq:def-P}
\end{align}
Using \eqref{eq:def-P}, we can write $g_t$ as
\begin{align*}
g_t = \eta Cw_t + \eta (C_t - C) Pw_t.
\end{align*}
Since $S_t$ is sampled uniformly at random, we have $E[C_t]=C$. Taking the expectation on the dot product of $u_k$ and \eqref{eq:VR-power-momentum}, we obtain
\begin{align*}
E[u_k^Tw_{t+1}] = 2(1-\eta+\eta\lambda_k) E[u_k^Tw_{t}] - \beta E[u_k^Tw_{t-1}].
\end{align*}
Since an optimal solution of the PCA problem \eqref{prob:pca} is the first eigenvector $u_1$ of the covariance matrix $C$, the optimality gap is measured as $\sum_{k=2}^d (u_k^Tw_t)^2 / (u_1^Tw_t)^2$, representing how closely $w_t$ is aligned with $u_1$ \cite{golub2012matrix}. Note that if $w_t = u_1$, this ratio is $0$. Our analysis studies it in expectation by providing a bound for $\sum_{k=2}^d E[(u_k^Tw_t)^2] / E[(u_1^Tw_t)^2]$.

Note that $1-\eta+\eta\lambda_k \geq 0$ holds for every $k$ since $\eta \in (0,1]$ and $\lambda_k \geq 0$. In the above, $2(1-\eta+\eta\lambda_k)$ corresponds to $2\lambda_k$ in the dot product of $u_k$ and \eqref{eq:Power-momentum}.
Taking the square of it, we define
\begin{align*}
\alpha_k(\eta) = 4(1-\eta+\eta \lambda_k)^2
\end{align*}
for $1 \leq k \leq d$. Analogous to the optimal momentum parameter $\beta = \lambda_2^2$ in \eqref{eq:Power-momentum}, we also consider
\begin{align*}
\beta(\eta) = {(1-\eta+\eta \lambda_2)^2}.
\end{align*}
To characterize the variance of $C_t-C$, we define constant $K$ as
\begin{align*}
K = \|E[(C_t-C)^2]\|.
\end{align*}
Furthermore, let $P_{i,j}(K)$ denote a polynomial having the form of
\begin{align*}
P_{i,j}(K) = \sum_{l=i}^j c_l K^l
\end{align*}
and let $p_t(\alpha,\beta)$ and $q_t(\alpha,\beta)$ be recurrence polynomials satisfying
\begin{align*}
p_t(\alpha,\beta) &= (\alpha - \beta) p_{t-1} (\alpha,\beta) - \beta (\alpha-\beta) p_{t-2} (\alpha,\beta) \\
& \quad + \beta^3 p_{t-3} (\alpha,\beta) \\
q_t(\alpha,\beta) &= (\alpha - \beta) q_{t-1} (\alpha,\beta) - \beta (\alpha-\beta) q_{t-2} (\alpha,\beta) \\
& \quad + \beta^3 q_{t-3} (\alpha,\beta)
\end{align*}
for $t \geq 3$ with
\begin{align*}
p_0(\alpha,\beta) = 1, \quad p_1(\alpha,\beta) = \frac{\alpha}{4}, \quad p_2(\alpha,\beta) =\Big( \frac{\alpha}{2}-\beta \Big)^2, \\
q_0(\alpha,\beta) = 1, \quad q_1(\alpha,\beta) = {\alpha}, \quad q_2(\alpha,\beta) = ( {\alpha}-\beta)^2.
\end{align*}

In the following Lemmas~\ref{lemma:uk-wt-square-exp-equality},~\ref{lemma:P-wt-square-inequality}, and~\ref{lemma:single-epoch-convergence}, we consider a single epoch, which corresponds to one inner loop iteration.
\begin{lemma}
\label{lemma:uk-wt-square-exp-equality}
For any $0 < \eta \leq 1$, we have
\begin{align*}
E[& (u_k^T w_t)^2] = p_{t}(\alpha_k(\eta),\beta(\eta)) E[(u_k^Tw_0)^2]  \\
& + 4\eta^2 \sum_{r=1}^{t-1} q_{t-r-1}(\alpha_k(\eta),\beta(\eta)) E[w_{r}^TPM_kPw_{r}]
\end{align*}
for $1 \leq k \leq d$.
\end{lemma}
In Lemma~\ref{lemma:uk-wt-square-exp-equality}, $E[(u_k^T w_t)^2]$ is decomposed into two parts. The first term originates from $E[(u_k^Tw_0)^2]$ while the second sum consisting of $E[w_{t}^TPM_kPw_{t}]$ stems from the stochastic variance of $g_t - \nabla f(w_t)$. Since the stochastic variance needs to be appropriately controlled to obtain convergence, we analyze it in the following lemma.
\begin{lemma}
\label{lemma:P-wt-square-inequality}
For any $0< \eta \leq 1$, we have
\begin{align*}
E[w_{t}^TPM_kPw_{t}] \leq \eta^2 {{P}_{1,t}(K)} \sum_{k=2}^d E \big[ (u_k^Tw_0)^2 \big]
\end{align*}
where $1 \leq k \leq d$.
\end{lemma}
Lemma~\ref{lemma:P-wt-square-inequality} provides an upper bound for $E[w_{t}^TPM_kPw_{t}]$ which is a function of $\eta$, $K$, and $\sum_{k=2}^d E[ (u_k^Tw_0)^2]$. The upper bound can be interpreted in the following way. If the step size $\eta$ or $\sum_{k=2}^d E[ (u_k^Tw_0)^2]$ is small, iterate updates are made near $w_0$, and therefore $w_t$ is closely aligned with $w_0$. Since this makes $g_t$ close to the exact gradient $\nabla f(w_t)$, the stochastic variance becomes small. On the other hand, when $K$ is small (or $|S|$ is large), the stochastic variance of $C_t - C$ becomes small, resulting in more accurate $g_t$.

The next lemma establishes the error bound in expectation within a single epoch.
\begin{lemma}
\label{lemma:single-epoch-convergence}
For any $0< \eta \leq 1$, we have
\begin{align*}
\frac{\sum_{k=2}^d E[(u_k^Tw_m)^2]}{E[(u_1^Tw_m)^2]} 
\leq \rho(\eta,K) \frac{\sum_{k=2}^d E[(u_k^Tw_0)^2]}{E[(u_1^Tw_0)^2]}
\end{align*}
where
\begin{align*}
\rho(\eta, K) = g(\eta) + h(\eta,K)
\end{align*}
and
\begin{align*}
g(\eta) = \frac{p_m(\alpha_2(\eta),\beta(\eta))}{p_m(\alpha_1(\eta),\beta(\eta))}, \quad h(\eta,K) = \eta^4 {{P}_{1,m}(K)}.
\end{align*}
Function $g(\eta)$ is a decreasing function of $\eta$ on $(0,1]$ and it equals to
\begin{align*}
g&(\eta) = \Bigg[ \frac{(1-\eta+\eta\lambda_1+\sqrt{2-2\eta+\eta(\lambda_1+\lambda_2)} \sqrt{\eta \Delta})^m}{2(1-\eta+\eta\lambda)^m} \nonumber \\
& - \frac{(1-\eta+\eta\lambda_1 - \sqrt{2-2\eta+\eta(\lambda_1+\lambda_2)} \sqrt{\eta \Delta})^m}{2(1-\eta+\eta\lambda)^m} \Bigg]^{-2}.
\end{align*}
Moreover, there exists some $0 < \bar{\eta}(K)$ such that for every $\eta \in (0,\bar{\eta}(K)]$, we have
\begin{align*}
\rho(1, 0) \leq \rho(\eta, K) < 1.
\end{align*}
\end{lemma}

While Lemmas~\ref{lemma:uk-wt-square-exp-equality},~\ref{lemma:P-wt-square-inequality} and \ref{lemma:single-epoch-convergence} deal with a single epoch, the next result establishes the convergence of the overall algorithm.

\begin{theorem}
\label{theorem:algorithm-convergence}
Suppose we execute Algorithm \ref{alg:VRPowerHB} for $s$ epochs starting from an initial unit vector $\tilde{w}_0$ such that $u_1^T\tilde{w}_0 \neq 0$. There exists some $0 < \bar{\eta}(K)$ such that for every $\eta \in (0,\bar{\eta}(K)]$, we have
\begin{align}
\frac{\sum_{k=2}^d E[(u_k^T \tilde{w}_s)^2]}{E[(u_1^T \tilde{w}_s)^2]} 
&\leq \rho(\eta, K)^s \bigg(\frac{1-(u_1^T\tilde{w}_0)^2}{(u_1^T\tilde{w}_0)^2}\bigg)
\label{eq:main-theorem}
\end{align}
where $0<\rho (\eta, K) <1$.
\end{theorem}

Lemma~\ref{lemma:single-epoch-convergence} provides a convergence rate for a single epoch where the rate of convergence $\rho(\eta,K)$ consists of the expected rate $g(\eta)$ and the additional variance term $h(\eta,K)$ arising from stochastic errors. Owing to the momentum term in \eqref{eq:VR-power-momentum}, $g(\eta)$ depends inversely on the square root of the eigen-gap $\Delta$, making it appealing when $\Delta$ is small. In order to benefit from the $\sqrt{\Delta}$ term, $\eta$ should not be too small. However, as increasing $\eta$ also enlarges $h(\eta,K)$, $K$ must be controlled in order to achieve stochastic acceleration. Since decreasing $K$ does not affect $g(\eta)$ but reduces the stochastic variance term $h(\eta,K)$, a large $\eta$ can be tolerated with a small $K$. However, even if $K$ is large, corresponding to a small mini-batch size, we can still obtain convergence by choosing a sufficiently small $\eta$. Since the stochastic variance term $h(\eta,K)$ vanishes very quickly as $\eta$ approaches $0$, we can always make $\rho(\eta,K)$ smaller than 1 by selecting a sufficiently small $\eta$ as showed in Lemma~\ref{lemma:single-epoch-convergence}. On the other hand, since $g(\eta)$ is a decreasing function of $\eta$, $\rho(\eta,K)$ is lower bounded by $\rho(1,0)$ for any $\eta \in (0,1]$.

Compared to VR-PCA and VR Power+M, VR HB Power is favorable since it works with any size of the mini-batch, yet enjoys acceleration when the mini-batch size is sufficiently large. Unlike VR Power+M, VR-PCA works with any size of the mini-batch size but does not accelerate when the size of the mini-batch is large since its rate does not depend inversely on the square root of $\Delta$. On the other hand, although the convergence rate of VR Power+M has an inverse dependency on the square root of $\Delta$, it lacks the measure to control the progress of the algorithm, making it fail to converge when the mini-batch size is not sufficiently large. Moreover, even when the mini-batch size is large, its performance cannot be superior to VR HB Power since it is a special case of VR HB Power with $\eta$ being $1$. By appropriately choosing the step size $\eta$, VR HB Power performs no worse than VR Power+M in all cases. Table \ref{table:summary-convergence-rate} summarizes the convergence rates of these algorithms. Note that our convergence rate for VR Power+M is established by setting $\eta = 1$ in the rate for VR HB Power. The rate for VR-PCA is derived by using the recurrence polynomials without the momentum term.

Theorem \ref{theorem:algorithm-convergence} provides a convergence result for the entire algorithm. Specifically, it establishes the global linear convergence of the algorithm where the ratio of the expectation of the components orthogonal to $u_1$ to the expectation of the component aligned with $u_1$ goes to zero at a rate of $\rho(\eta,K)<1$. Note that the rate of convergence $\rho(\eta,K)$ does not depend on the epoch but is kept the same across the epochs. Although VR-PCA and VR Power+M practically converge regardless of the initial iterate $\tilde{w}_0$, there has been no analysis proving their global convergence. However, by following the techniques developed in this section, their global convergence is proved by our framework for any initial iterate $\tilde{w}_0$ and in expectation (as opposed to the probabilistic statements in \cite{shamir2015stochastic} and \cite{xu2018accelerated}).

Note that the condition $\eta \in (0, \bar{\eta}(K)]$ and $\rho = \rho(\eta,K)$ can be changed to there exists $0<\eta_2(K)$ such that for every $\eta_1 \in (0, \eta_2(K))$ and $\eta \in [\eta_1, \eta_2(K)]$, inequality \eqref{eq:main-theorem} holds for $\rho = \rho(\eta_1,K)$. This can be easily established based on the proof.

\section{Practical Considerations}
\label{sec:practical-considerations}
In this section, we discuss some practical considerations of the algorithm. 
First, to make the algorithm numerically stable, we consider
\begin{align*}
w_t \leftarrow w_t /\|w_{t+1}\|_2, \quad w_{t+1} \leftarrow w_{t+1} /\|w_{t+1}\|_2
\end{align*}
at the end of the inner loop after updating $w_{t+1}$ as introduced in \cite{xu2018accelerated}. Since the above scaling scheme does not impact the sample path of $w_t/\|w_t\|$, the same result can be obtained with numerical stability.

Another important issue is the estimation of $\lambda_2$, which is involved in determining the value of the momentum parameter. Using the two-loop structure of the algorithm and the fact that the gradients of the outer-loop iterates $\tilde{w}_s$ are exactly computed, we estimate $\lambda_2$ using the two consecutive outer-loop iterates $\tilde{w}_{s-1}$ and $\tilde{w}_{s}$ at a regular interval.

Using the Rayleigh quotient and the second eigenvector $u_2$ of the covariance matrix $C$, the second eigenvalue $\lambda_2$ can be expressed as 
\begin{align}
\lambda_2 = \frac{u_2^TCu_2}{u_2^T u_2}.
\label{eq:lambda2}
\end{align}
In deterministic power iteration and its variants, an outer-iterate $\tilde{w}_s$ first approaches the subspace spanned by $u_1$ and $u_2$ before converging to $u_1$.
After a number of outer-iterations, vector $\tilde{w}_s$ can be approximated by a linear combination of $u_1$ and $u_2$ and the component of $u_1$ becomes dominant as the iterations proceed.

Based on this observation, we estimate $u_2$ using two consecutive outer-loop iterates $\tilde{w}_s$ and $\tilde{w}_{s-1}$ as
\begin{align}
\hat{u}_{2,s} &= \tilde{w}_{s-1} - (\tilde{w}_{s-1}^T \tilde{w}_{s}) \tilde{w}_{s}.
\label{eq:u2-hat}
\end{align}
The idea of the above estimation is to project $\tilde{w}_{s-1}$ to the space orthogonal to $\tilde{w}_{s}$.
If $\tilde{w}_{s} \approx u_1$ and $\tilde{w}_{s-1} \approx \alpha_1 u_1 + \alpha_2 u_2$ for some $\alpha_1, \alpha_2(\neq 0)$, we have $\hat{u}_{2,s} \approx u_2$. 
By substituting $u_2$ with $\hat{u}_{2,s}$ in \eqref{eq:lambda2}, we obtain
\begin{align}
\hat{\lambda}_{2,s} = & \frac
{\tilde{w}_{s-1}^TC\tilde{w}_{s-1} - 2 \theta_s \tilde{w}_s^TC \tilde{w}_{s-1} + \theta_s^2 \tilde{w}_s^TC\tilde{w}_s}{1-\theta_s^2}
\label{eq:lambda2-hat}
\end{align}
where 
\begin{align*}
\theta_s = \tilde{w}_{s-1}^T\tilde{w}_{s}.
\end{align*}
While two matrix-vector multiplications, $C\tilde{w}_{s-1}$ and $C\tilde{w}_{s}$, are involved in computing \eqref{eq:lambda2-hat}, they incur no extra computation since they are the exact gradients of $\tilde{w}_{s-1}$ and $\tilde{w}_{s}$, which are computed regardless of the estimation. As a result, we can obtain $\hat{\lambda}_2$ by only computing inner products. This update is repeated at the start of each outer-loop iteration after computing $\tilde{g}$ followed by setting the momentum parameter $\beta_s$ at outer iteration $s$ as
\begin{align*}
\beta_s = (1-\eta+\eta \hat{\lambda}_{2,s})^2.
\end{align*}
\begin{figure*}[h]
\centering
\includegraphics[trim={4cm 0 3cm 1cm},clip,scale=0.39]{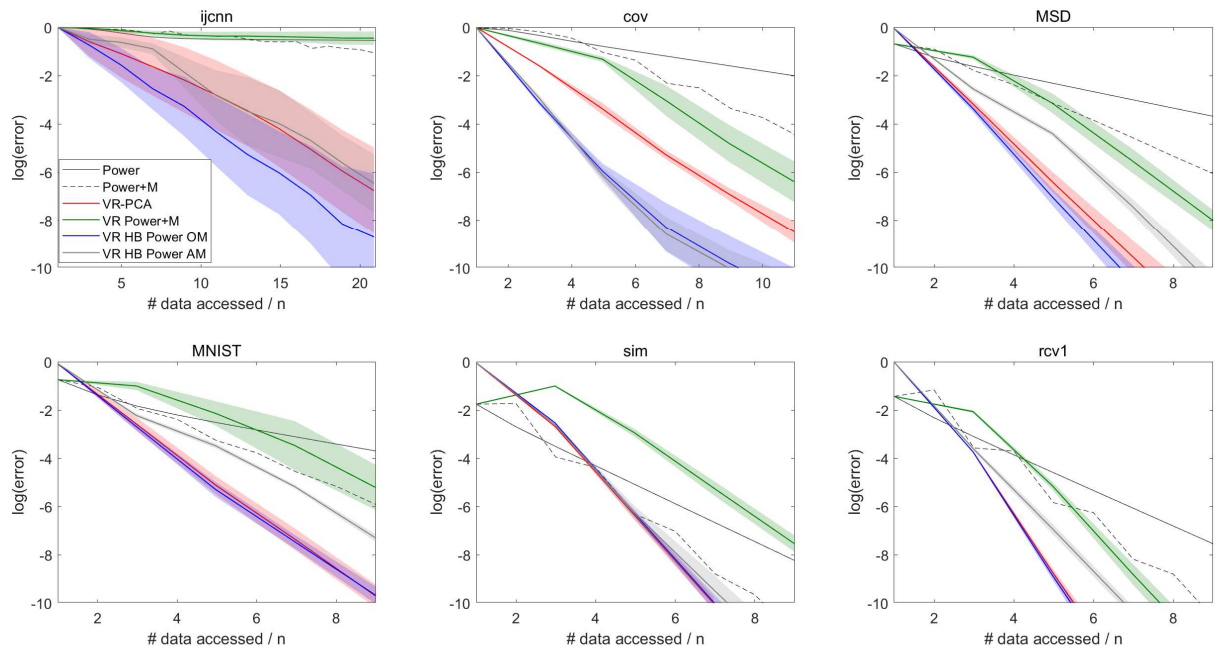}
\caption{Experimental Results (small-batch)}
\label{fig:small-batch}
\end{figure*}

\section{Numerical Experiments}
\label{sec:numerical-experiments}
\begin{figure*}[h]
\centering
\includegraphics[trim={4cm 0 3cm 1cm},clip,scale=0.39]{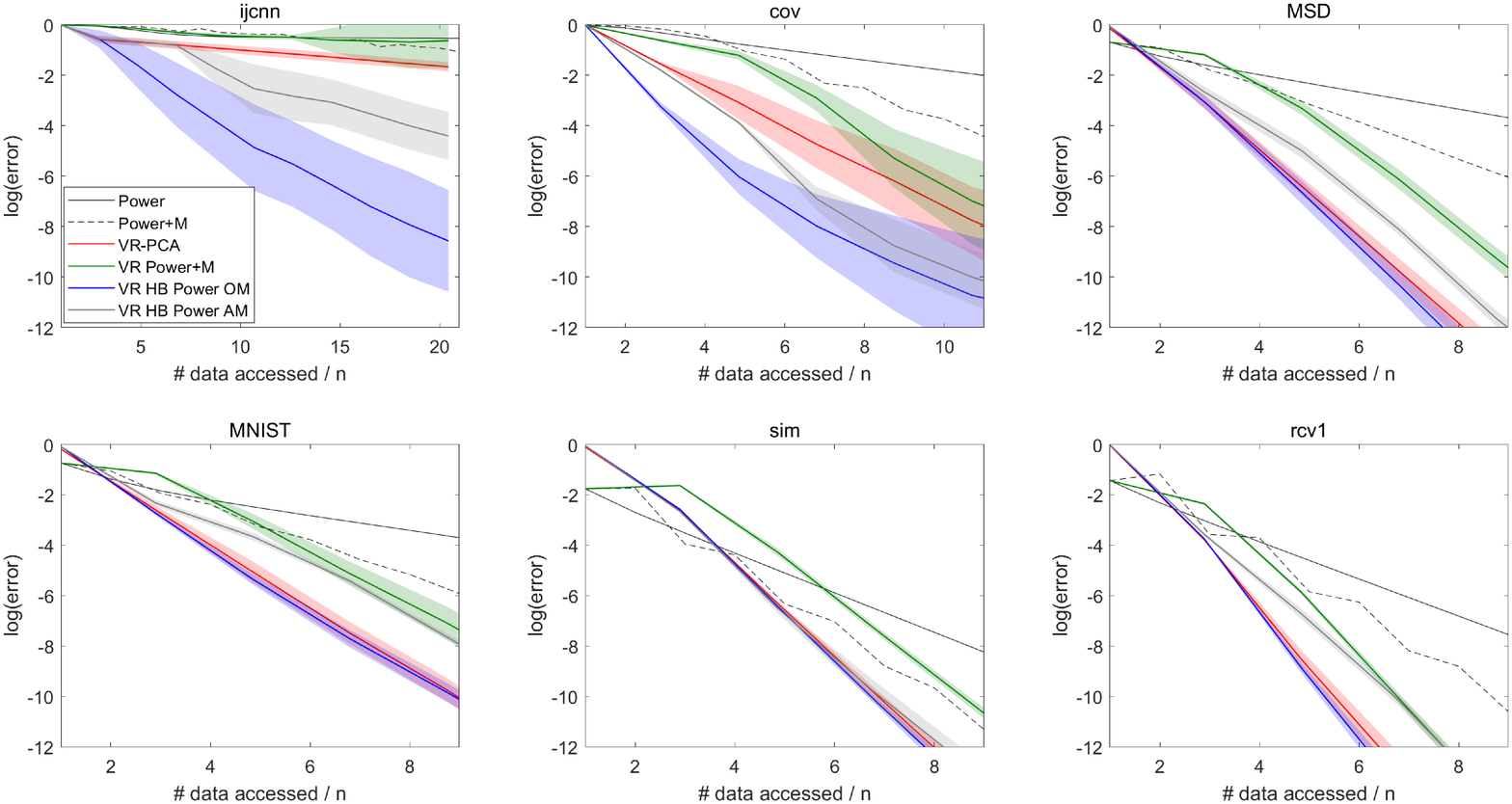}
\caption{Experimental Results (large-batch)}
\label{fig:large-batch}
\end{figure*}
In this section, we numerically compare the performance of VR HB Power with that of (i) Power, (ii) Power+M, (iii) VR-PCA, and (iv) VR Power+M on finding the first eigenvector $u_1$ of covariance matrix $C$ constructed by data vectors $a_i, i=1,\ldots,n$ using real world datasets. We include variants of power iteration and other algorithms are excluded due to the complexity of implementation.

\subsection{Datasets}
The datasets include ijcnn \cite{prokhorov2001ijcnn}, covertype \cite{blackard1999comparative}, YearPredictionMSD \cite{Bertin-Mahieux2011}, MNIST \cite{lecun1998gradient}, real-sim and rcv1 \cite{lewis2004rcv1} as summarized in Tabel~\ref{table:datasets}. All of them are obtained either from the UCI repository \cite{Dua:2017} or the LIBSVM library \cite{chang2011libsvm} and they are carefully chosen to incorporate a variety of datasets in terms of size and eigen-gap.
\begin{table}[h]
\vspace{-0.25cm}
\caption{Datasets}
\label{table:datasets}
\begin{center}
\begin{small}
\begin{sc}
\begin{tabular}{lrrrc}
\toprule
dataset           & $n$ & $d$ & sparsity & $\lambda_2 / \lambda_1$ \\ \midrule
icjnn(test)             & 91,701           & 22   & 59.09 \% & 0.9921  \\ 
cov         & 581,012          & 54     & 22.00 \%  & 0.7894  \\
MSD               & 463,715          & 90 & 100.00 \% & 0.6776  \\
MNIST             & 70,000           & 764  & 1.96 \% & 0.7167  \\
sim         & 72,309           & 20,958  & 0.24 \% & 0.4053 \\
rcv1              & 804,414          & 47,236   & 0.16 \% & 0.4289  \\
\bottomrule
\end{tabular}
\end{sc}
\end{small}
\end{center}
\vskip -0.1in
\end{table}
The first four datasets are standardized with a mean of zero and standard deviation of one while the last two datasets are scaled to range between $0$ and $1$ to preserve their sparsity.

\subsection{Settings}
Since we consider the mini-batch setting of variance reduction methods (VR-PCA, VR Power+M, VR HB Power) which have the two-loop structure, it is necessary to choose the epoch length $m$ and the mini-batch size $|S|$. For the choice of $m$ and $|S|$, it is common to select $m$ and $|S|$ such that $m \cdot |S|=n$. Following this principle, we consider $|S|=1\% \cdot n$ and $m = 100$ for the small-batch case and $|S|=5\% \cdot n$ and $m = 20$ for the large-batch case. For the momentum parameter $\beta$ in VR HB Power, Power+M, and VR Power+M, we utilize the true value of $\lambda_2$ for Power+M and VR Power+M and consider both the true value of $\lambda_2$ (VR HB Power OM) and the adaptive estimation procedure of $\lambda_2$ (VR HB Power AM) presented in Section~\ref{sec:practical-considerations} for VR HB Power. For numerical stability, the scaling scheme from Section~\ref{sec:practical-considerations} is also used for VR HB Power and VR Power+M. To find the best performance, the step sizes $\eta$ in VR HB Power OM and VR-PCA are chosen using grid search and the step size in VR HB Power AM is set to that of VR HB Power OM.

\subsection{Results}
Figure~\ref{fig:small-batch} and Figure~\ref{fig:large-batch} display the experimental results for the small and large batch cases, respectively. In the figures, the x-axis represents the number of data points accessed over the number of total data points and the y-axis represents the error gap, $1-(\tilde{w}_s^Tu_1)^2$, in the log-scale. For each case, the stochastic algorithms are repeated 10 times by varying the random seed. The lines represent their means and the values within one standard deviation away from the means are shaded.

As shown in the figures, with the true value of $\lambda_2$, VR HB Power OM consistently outperforms the other algorithms in both cases. Particularly, it surpasses the other algorithms by a large margin when $\lambda_2/\lambda_1$ is close to $1$ as seen in the cases of ijcnn and cov. If $\lambda_2/\lambda_1$ is not large, its performance is similar to that of VR-PCA in the small-batch setting and slightly better in the large-batch setting as shown in the cases of MNIST, sim, and rcv1. While VR-PCA is competitive to VR HB Power OM when $\lambda_2/\lambda_1$ and $|S|$ are small, VR Power+M always falls behind VR-PCA and VR HB Power OM. Moreover, it fails to converge in the small batch setting of ijcnn and is sometimes inferior to Power and Power+M.

On other hand, VR HB Power AM does not have the same performance as VR HB Power OM. However, as $\hat{\lambda}_{2,s}$ approaches $\lambda_2$, it asymptotically attains the same rate of convergence. If $\lambda_2/\lambda_1$ is close to $1$, slow progress at the start can be compensated by asymptotic performance as $\hat{\lambda}_{2,s}$ approaches $\lambda_2$. Also, the estimation of $\lambda_2$ becomes stable if the batch size is large. Therefore, it exhibits superior performance to VR-PCA and VR Power+M when $\lambda_2/\lambda_1$ is close to $1$ and the size of the mini-batch is large as seen in the cases of ijcnn and cov in Figure~\ref{fig:large-batch}.

\section{Conclusion}
In this paper, we present a stochastic heavy ball power iteration algorithm for solving PCA and present a convergence analysis for it. By incorporating a step size, the presented algorithm works with any size of the mini-batch and attains acceleration if the mini-batch size is large, making it attractive in parallel settings. In the convergence analysis, we show that our algorithm attains global linear convergence to the first eigenvector of the covarince matrix in expectation. This result is analogous to those of stochastic variance-reduced gradient methods for convex optimization but has been never shown for previous stochastic power iteration algorithms since it requires completely different techniques. We stress that our analysis is the first such analysis and its framework can be applied to analyze previous stochastic power iteration algorithms, showing their global linear convergence in expectation. The experimental results show that if $\lambda_2$ is known, our algorithm consistently outperforms other algorithms, especially when the eigen-gap is small. Even if $\lambda_2$ is unknown, our algorithm can be run with the adaptive estimation procedure of $\lambda_2$. The numerical experiments exhibit that it still outperforms the previous stochastic algorithms if the eigen-gap is small and the size of the mini-batch is large.

\color{black}
\newpage
\bibliography{example_paper}
\bibliographystyle{icml2019}

\appendix

\onecolumn
\section{Supplimentary Material}
\subsection{Main Results}
\begin{proof}[Proof of Lemma~\ref{lemma:uk-wt-square-exp-equality}]
From
\begin{align*}
w_1 &= (1- \eta)w_0 + \eta \tilde{g} \\
&= (1- \eta)w_0 + \eta Cw_0,
\end{align*}
we have
\begin{align}
u_k^Tw_1 &= (1-\eta) u_k^Tw_0 + \eta u_k^TC w_0 \nonumber \\
&= (1-\eta) u_k^Tw_0 + \eta \lambda_ku_k^Tw_0 \nonumber \\
&= (1-\eta+\eta \lambda_k) u_k^Tw_0. \label{eq:uk_w1}
\end{align}
Taking the expectation of the square of \eqref{eq:uk_w1}, we obtain
\begin{align}
E[(u_k^Tw_1)^2] = (1-\eta+\eta \lambda_k)^2 E[(u_k^Tw_0)^2] = \frac{\alpha_k(\eta)}{4} E[(u_k^Tw_0)^2].
\label{eq:uk-w1-square-exp}
\end{align}
Next, from \eqref{eq:VR-power-momentum}, we have
\begin{align}
w_{t+1} &= 2 \bigg( (1-\eta) w_{t} + \eta \frac{1}{|S_{t}|} \sum_{i_{t} \in S_{t}} a_{i_{t}}a_{i_{t}}^T \Big( w_{t}- \frac{(w_{t}^T w_0)}{\|w_0\|^2} w_0 \Big) + \frac{(w_{t}^T w_0)}{\|w_0\|^2} \tilde{g} \bigg) - \beta(\eta) w_{t-1} \nonumber \\ 
&= 2 \bigg( (1-\eta) w_{t} + \eta \frac{1}{|S_{t}|} \sum \limits_{i_{t} \in S_{t}} a_{i_{t}}a_{i_{t}}^T \Big( I-\frac{w_0 w_0^T}{\|w_0\|^2} \Big) w_{t} + C \frac{w_0 w_0^T}{\|w_0\|^2} w_t \bigg) - \beta(\eta) w_{t-1} \nonumber \\ 
&= 2 \bigg( (1-\eta) w_{t} + \eta Cw_{t} + \eta \frac{1}{|S_{t}|} \sum_{i_{t} \in S_{t}} (a_{i_{t}} a_{i_{t}}^T - C)\Big( I-\frac{w_0w_0^T}{\|w_0\|^2} \Big) w_{t} \bigg) - \beta(\eta) w_{t-1} \nonumber \\
&= 2 \big( (1-\eta) w_{t} + \eta Cw_{t} + \eta (C_t - C)P w_{t} \big) - \beta(\eta) w_{t-1}, \label{eq:VR-HVB-wt}
\end{align}
leading to
\begin{align}
u_k^Tw_{t+1} = 2 \big( (1-\eta+\eta \lambda_k) u_k^Tw_{t} + \eta u_k^T (C_{t}-C) P w_{t} \big) - \beta(\eta) u_k^Tw_{t-1}. \label{eq:uk-wt}
\end{align}
Taking the square of \eqref{eq:uk-wt}, we have
\begin{align}
(u_k^Tw_{t+1})^2 &= 4(1-\eta+\eta \lambda_k)^2 (u_k^Tw_{t})^2 + 4\eta^2 w_{t}^T P (C_{t}-C) u_k u_k^T (C_{t}-C) P w_{t} + (\beta(\eta))^2 (u_k^Tw_{t-1})^2 \nonumber \\
& \quad + 8 \eta (1-\eta+\eta \lambda_k) u_k^T w_{t} u_k^T (C_{t}-C)Pw_{t} - 4(1-\eta+\eta \lambda_k) \beta(\eta) u_k^Tw_{t}u_k^Tw_{t-1} \nonumber \\
& \quad - 4\eta \beta(\eta) u_k^T(C_{t}-C)Pw_{t} u_k^T w_{t-1}. \label{eq:VR-HVB-uk-wt-square}
\end{align}
Since $S_t$ is sampled uniformly at random, $C_t$ is independent of $S_1, \ldots, S_{t-1}$ and identically distributed with $E[C_t]=C$. Therefore,
\begin{align*}
E[u_k^Tw_{t}(C_{t}-C)Pw_{t}] 
= E[E[ u_k^Tw_{t} u_k^T(C_{t}-C)Pw_{t} | w_0, S_1, \ldots, S_{t-1}]] = E[u_k^Tw_{t} u_k^T E[C_{t}-C] Pw_{t}] = 0.
\end{align*}
Similarly, we have
\begin{align}
E[u_k^T(C_{t}-C)Pw_{t} u_k^T w_{t-1}] = 0.
\label{eq:exp-0}
\end{align}
As a result, we obtain
\begin{align}
E[(u_k^Tw_{t+1})^2] &= \alpha_k(\eta) E[(u_k^T w_{t})^2] - 2 \sqrt{\alpha_k(\eta)} \beta(\eta) E[(u_k^Tw_{t})(u_k^Tw_{t-1})] + (\beta(\eta))^2 E[(u_k^T w_{t-1})^2] \nonumber \\
&\quad + 4 \eta^2 E[w_{t}^T P M_k P w_{t}]. \label{eq:uk-wt-square-1}
\end{align}
Using \eqref{eq:uk_w1} and \eqref{eq:uk-w1-square-exp} for $t=1$ in \eqref{eq:uk-wt-square-1}, we have
\begin{align}
E[(u_k^Tw_2)^2] = \Big( \frac{\alpha_k(\eta)}{2} - \beta(\eta) \Big)^2 E[(u_k^T w_0)^2] + 4\eta^2 E[w_{1}^T P M_k P w_{1}]. \label{eq:uk_w2-square}
\end{align}
Moreover, by using \eqref{eq:uk-wt} with $t-1$, multiplying it with $u_k^Tw_{t-1}$, taking expectation and using \eqref{eq:exp-0} with $w_t$ being $w_{t-1}$ (which can be derived in the same way as \eqref{eq:exp-0}) , we have
\begin{align}
\label{eq:uk-wt-wt-1}
E[(u_k^Tw_{t})(u_k^Tw_{t-1})] = \sqrt{\alpha_k(\eta)} E[(u_k^T w_{t-1})^2] - \beta(\eta) E[(u_k^T w_{t-1})(u_k^T w_{t-2})].
\end{align}
Using \eqref{eq:uk-wt-wt-1}, we can further write \eqref{eq:uk-wt-square-1} as
\begin{align}
E[(u_k^Tw_{t+1})^2] &= \alpha_k(\eta) E[(u_kw_{t})^2] - 
\beta(\eta) ( 2\alpha_k(\eta) - 
\beta(\eta) ) E[(u_k^T w_{t-1})^2] \nonumber \\
&\quad + 2 \sqrt{\alpha_k(\eta)} (\beta(\eta))^2  E[(u_k^T w_{t-1})(u_k^T w_{t-2})] + 4 \eta^2 E[w_{t}^T P M_k P w_{t}]. \label{eq:uk-wt-square-2}
\end{align}
With $t-1$ in \eqref{eq:uk-wt-square-1}, we have
\begin{align}
E[(u_k^Tw_{t})^2] &= \alpha_k(\eta) E[(u_k^T w_{t-1})^2] - 2 \sqrt{\alpha_k(\eta)} \beta(\eta)  E[(u_k^Tw_{t-1})(u_k^Tw_{t-2})] + (\beta(\eta))^2 E[(u_k^T w_{t-2})^2] \nonumber \\
& \quad + 4 \eta^2 E[w_{t-1}^T P M_k P w_{t-1}]. \label{eq:uk-wt-1-square-1}
\end{align}
Adding \eqref{eq:uk-wt-1-square-1} multiplied by $\beta(\eta)$ to \eqref{eq:uk-wt-square-2}, we obtain
\begin{align}
E[(u_k^Tw_{t+1})^2] & = (\alpha_k(\eta) - \beta(\eta)) E[(u_k^Tw_{t})^2] - \beta(\eta) (\alpha_k(\eta) - \beta(\eta)) E[(u_k^T w_{t-1})^2] + (\beta(\eta))^3 E[(u_k^Tw_{t-2})^2] \nonumber \\
& \quad + 4\eta^2 E[w_{t}^TPM_kPw_{t}] + 4\eta^2 \beta(\eta) E[w_{t-1}^TPM_kPw_{t-1}].
\label{eq:uk-wt-square}
\end{align}
With $t-1$ in \eqref{eq:uk-wt-square}, we finally have
\begin{align}
E[(u_k^Tw_{t})^2] &= (\alpha_k(\eta) - \beta(\eta)) E[(u_k^T w_{t-1})^2] - \beta(\eta) (\alpha_k(\eta) - \beta(\eta)) E[(u_k^T w_{t-2})^2] + (\beta(\eta))^3 E[(u_k^T w_{t-3})^2] \nonumber \\
& \quad + 4\eta^2 E[w_{t-1}^T P M_k P w_{t-1}] + 4\eta^2 \beta(\eta) E[w_{t-2}^T P M_k P w_{t-2}]
\label{eq:uk-wt-square-exp-t-3}
\end{align}
for $t \geq 3$.

\noindent Using Lemma~\ref{lemma:coefficient-inequality} for $E[(u_k^Tw_{t})^2]$ defined by \eqref{eq:uk-w1-square-exp}, \eqref{eq:uk_w2-square}, and \eqref{eq:uk-wt-square-exp-t-3} with
\begin{align*}
\alpha = \alpha_k(\eta), \quad \beta = \beta(\eta), \quad L_0 = E[(u_k^Tw_0)^2], \quad L_t = 4\eta^2 E[w_{t}^TPM_kPw_{t}],
\end{align*}
we have
\begin{align*}
E[(u_k^Tw_t)^2] = p_t(\alpha_k(\eta),\beta(\eta)) E[(u_k^Tw_0)^2] + 4\eta^2 \sum_{r=1}^{t-1} q_{t-r-1}(\alpha_k(\eta),\beta(\eta)) E[w_{r}^TPM_kPw_{r}].
\end{align*}
\end{proof}

\begin{proof}[Proof of Lemma~\ref{lemma:P-wt-square-inequality}]
From Lemma~\ref{lemma:trace}, we have
\begin{align}
E[w_{t}^TPM_kPw_{t}] \leq E[\|M_k\|] E[\|Pw_{t}\|^2] \leq K E[\|Pw_{t}\|^2].
\label{eq:trace-bound}
\end{align}
By the definition of $P$ in \eqref{eq:def-P}, we have
\begin{align}
E[\|Pw_0\|^2] = E \bigg[ \bigg\| \bigg( I- \frac{w_0w_0^T}{\|w_0\|^2} \bigg ) w_0 \bigg\|^2 \bigg] = E[\|w_0 - w_0\|^2] = 0.
\label{eq:P-w0-square-exp}
\end{align}
Using Lemma~\ref{lemma:quadratic-form}, we obtain
\begin{align}
E [\| P w_1 \|^2] &= E \bigg[ \bigg\| \bigg(I - \frac{w_0 w_0^T}{\|w_0\|^2} \bigg) \Big( \eta w_0 + \eta C w_0 \Big) \bigg\|^2 \bigg] \nonumber \\
&= E \bigg[ \bigg \| \eta C w_0 - \eta \frac{w_0 w_0^T C w_0}{\|w_0\|^2} \bigg\|^2 \bigg] \nonumber \\
&= \eta^2 E \bigg[ \|w_0\|^2 \bigg( \frac{w_0^T C^2 w_0}{\|w_0\|^2} - \frac{(w_0^T C w_0)^2}{\|w_0\|^4} \bigg) \bigg] \nonumber \\
&\leq 2 \eta^2 E \bigg[  \|w_0\|^2 \bigg( \lambda_1^2 - \lambda_1^2 \frac{(u_1^Tw_0)^2}{\|w_0\|^2} \bigg) \bigg] \nonumber \\
&= 2 \eta^2 \lambda_1^2 \sum_{k=2}^d E \big[ (u_k^Tw_0)^2 \big] \label{eq:P-w1-square-exp}
\end{align}
where we have used the fact that $u_1,\ldots,u_d$ form an orthonormal basis for the last equality.

For $t \geq 2$, we consider
\begin{align}
Pw_t = 2 \big( P((1-
\eta) I + \eta C)w_{t-1} + \eta P(C_t-C)Pw_{t-1} \big) - \beta(\eta) Pw_{t-2}.
\label{eq:P-wt}
\end{align}
Taking the squared norm of \eqref{eq:P-wt}, we have
\begin{align}
\| Pw_t \|^2 &= \| 2 P((1-
\eta) I + \eta C)w_{t-1} - \beta(\eta) Pw_{t-2} \|^2 + 4 \eta^2 \| P(C_t-C)Pw_{t-1} \|^2 \nonumber \\
& \quad + 4 \eta ( 2 P((1-
\eta) I + \eta C)w_{t-1} - \beta(\eta) Pw_{t-2} )^T P (C_t-C)Pw_{t-1}.\label{eq:VR-HVB-Pw-projection}
\end{align}
Similarly to \eqref{eq:exp-0}, we have
\begin{align*}
E[( P((1-
\eta) I + \eta C)w_{t-1} - \beta(\eta) Pw_{t-2} )^T P (C_t-C)Pw_{t-1}] = 0,
\end{align*}
resulting in
\begin{align}
E[\| Pw_t \|^2] = E[\| 2P((1-
\eta) I + \eta C)w_{t-1} - \beta(\eta) Pw_{t-2} \|^2] + 4 \eta^2 E[\| P(C_t-C)Pw_{t-1} \|^2]. \label{eq:VR-HVB-Pw-projection}
\end{align}
By the triangle inequality and $(a+b)^2 \leq 2(a^2+b^2)$ for all $a,b$, we have
\begin{align}
E[ \| 2P((1-
\eta) I + \eta C)w_{t-1} - \beta(\eta) Pw_{t-2} \|^2] &\leq E[ ( \| 2P((1-
\eta) I + \eta C)w_{t-1} \| + \beta(\eta) \| Pw_{t-2} \|)^2] \nonumber \\
&\leq 2 E[ \| 2P((1-
\eta) I + \eta C)w_{t-1} \|^2] + 2 (\beta(\eta))^2 E[ \| Pw_{t-2} \|^2]. \label{eq:VR-HVB-recurrence-Pw-2}
\end{align}
Using the definition of the spectral norm, we further have
\begin{align}
E [\| 2P((1-
\eta) I + \eta C)w_{t-1} \|^2] &= 4 E[\| ((1-\eta)I+ \eta C)P w_{t-1} + \eta(PC-CP) w_{t-1} \|^2] \nonumber \\
&\leq 8 E[\| ((1-\eta)I+ \eta C)P w_{t-1} \|^2] + 8 E[\| \eta(PC-CP) w_{t-1} \|^2] \nonumber \\
&\leq 8 E[\| (1-\eta) I+ \eta C \|^2 \| P w_{t-1} \|^2] + 8 \eta^2 E[\| PC-CP \|^2 \| w_{t-1} \|^2] \nonumber \\
&\leq 8 (1-\eta+\eta \lambda_1)^2 E[\| P w_{t-1} \|^2] + 8 \eta^2 E[\| PC-CP \|^2 \| w_{t-1} \|^2]. \label{eq:VR-HVB-recurrence-Pw-4}
\end{align}
From Lemmas~\ref{lemma:linear-algebra} and~\ref{lemma:quadratic-form}, we have
\begin{align}
E[\| PC-CP \|^2 \| w_{t-1} \|^2] &\leq E \bigg[ \bigg( \frac{w_0^TC^2w_0}{\| w_0\|^2} - \frac{(w_0^TCw_0)^2}{\| w_0\|^4} \bigg) \| w_{t-1} \|^2 \bigg] \nonumber \\
&\leq E \bigg[ 2 \bigg( \lambda_1^2 - \lambda_1^2 \frac{(u_1^Tw_0)^2}{\|w_0\|^2} \bigg) \|w_{t-1}\|^2 \bigg] \nonumber \\
&= E \bigg[ 2\lambda_1^2 \Big( \|w_0\|^2 - (u_1^Tw_0)^2 \Big) \frac{\| w_{t-1} \|^2}{\| w_0 \|^2} \bigg] \nonumber \\
&=  2\lambda_1^2 E \bigg[ \frac{\| w_{t-1} \|^2}{\| w_0 \|^2} \sum_{k=2}^d (u_k^Tw_0)^2  \bigg] \label{eq:VR-HVB-recurrence-Pw-5}
\end{align}
where we again have used the fact that $u_1,\ldots,u_d$ form an orthonormal basis for the last equality.

On the other hand, by observing that $P^2=P$, we have
\begin{align}
E [ \| P(C_t-C)Pw_{t-1} \|^2]
= E [ w_{t-1}^TP M_P Pw_{t-1}] 
\leq E[\| M_P \|] E [ \| P w_{t-1} \|^2] 
\leq K E [ \| P w_{t-1} \|^2]. \label{eq:VR-HVB-recurrence-Pw-3}
\end{align}
Using \eqref{eq:VR-HVB-recurrence-Pw-2}, \eqref{eq:VR-HVB-recurrence-Pw-4}, \eqref{eq:VR-HVB-recurrence-Pw-5}, \eqref{eq:VR-HVB-recurrence-Pw-3} in \eqref{eq:VR-HVB-Pw-projection}, we obtain
\begin{align}
E[ \| P w_t \|^2] & \leq  ( 8(1-\eta+\eta \lambda_1)^2 + 4\eta^2 K ) E [ \| P w_{t-1} \|^2] + 2 (\beta(\eta))^2 E [ \|  Pw_{t-2} \|^2] + 16 \eta^2 \lambda_1^2 E \bigg[ \frac{\| w_{t-1} \|^2}{\| w_0 \|^2} \sum_{k=2}^d (u_k^Tw_0)^2  \bigg]. \label{eq:P-wt-square-exp}
\end{align}
Using Lemma~\ref{lemma:three-term-inequality-compact-form} for $E[ \| P w_t \|^2]$ defined by \eqref{eq:P-w0-square-exp}, \eqref{eq:P-w1-square-exp}, and \eqref{eq:P-wt-square-exp} with
\begin{align*}
\alpha = 8(1-\eta+\eta \lambda_1)^2 + 4\eta^2 K, \quad \beta = 2 (\beta(\eta))^2, \quad L_t = 16 \eta^2 \lambda_1^2 E \bigg[ \frac{\| w_{t} \|^2}{\| w_0 \|^2} \sum_{k=2}^d (u_k^Tw_0)^2  \bigg],
\end{align*}
we obtain
\begin{align}
E[ \| P w_t \|^2] &\leq 16 \eta^2 \lambda_1^2 \sum_{l=0}^{t-1} r_{t-l-1} \big( 8(1-\eta+\eta \lambda_1)^2 + 4 \eta^2 K,2(\beta(\eta))^2 \big)  E \bigg[ \frac{\| w_{l} \|^2}{\| w_0 \|^2} \sum_{k=2}^d (u_k^Tw_0)^2  \bigg].
\label{eq:Pwt-squared}
\end{align}
For $0 \leq t-l-1 \leq t-1$, we have
\begin{align}
r_{t-l-1} \big( 8(1-\eta+\eta \lambda_1)^2 + 4\eta^2 K,2(\beta(\eta))^2 \big) \leq P_{0,t-l-1} (K)
\label{eq:Pwk-c1}
\end{align}
where we use the fact that the power of $\eta$ is bounded by $1$ since $\eta \in (0,1]$.

Moreover, letting $\bar{c}$ be defined as
\begin{align*}
\bar{c} = \text{max}_{i=1,\ldots,n} \frac{\text{max}_{|S_t|=i} \|C_l-C\|^2}{\|E[(C_l-C)^2]\|},
\end{align*}
we have
\begin{align*}
\text{max} \|C_l-C\|^2 \leq \bar{c} K.
\end{align*}

Since
\begin{align*}
\bigg\| 
\begin{bmatrix}
2 \big( (1-\eta)(I+\eta C)+\eta(C_l-C)P \big) & -\beta(\eta)I \\
I & 0 \\
\end{bmatrix}
\bigg\|^2
& \leq
2
\bigg\| 
\begin{bmatrix}
2(1-\eta)(I+\eta C) & -\beta(\eta)I \\
I & 0 \\
\end{bmatrix}
\bigg\|^2
+
8\eta^2
\bigg\|
\begin{bmatrix}
(C_l-C)P & 0 \\
0 & 0 \\
\end{bmatrix}
\bigg\|^2
\end{align*}
and
\begin{align*}
\bigg\|
\begin{bmatrix}
(C_l-C)P & 0 \\
0 & 0 \\
\end{bmatrix}
\bigg\|^2
= 
\big\|
(C_l-C)P
\big\|^2
\leq
\big\|
C_l-C
\big\|^2
\leq
\bar{c}K,
\end{align*}
we have
\begin{align*}
\bigg\| 
\begin{bmatrix}
2 \big( (1-\eta)(I+\eta C)+\eta(C_l-C)P \big) & -\beta(\eta)I \\
I & 0 \\
\end{bmatrix}
\bigg\|^2
&\leq P_{0,1}(K).
\end{align*}
Taking the the squared Euclidean norm on the both sides of
\begin{align*}
\begin{bmatrix}
w_l \\ w_{l-1}
\end{bmatrix}
=
\begin{bmatrix}
2\big( (1-\eta)(I+\eta C)+\eta(C_{ld-1}-C)P \big) & -\beta(\eta)I \\
I & 0 \\
\end{bmatrix}
\ldots
\begin{bmatrix}
2\big( (1-\eta)(I+\eta C)+\eta(C_0-C)P \big) & -\beta(\eta)I \\
I & 0 \\
\end{bmatrix}
\begin{bmatrix}
w_0 \\ 0
\end{bmatrix}
,
\end{align*}
we have
\begin{align}
\| w_l \|^2 \leq \| (w_l,w_{l-1}) \|^2 \leq P_{0,l}(K) \|w_0\|^2.
\label{eq:Pwk-c2}
\end{align}
Using \eqref{eq:Pwk-c1} and \eqref{eq:Pwk-c2} for \eqref{eq:Pwt-squared}, we have
\begin{align}
E[ \| P w_t \|^2] \leq 16 \eta^2 \lambda_1^2 \sum_{l=0}^{t-1} P_{0,t-l-1}(K) P_{0,l}(K)  E \bigg[ \sum_{k=2}^d (u_k^Tw_0)^2  \bigg] \leq \eta^2 P_{0,t-1}(K) E \bigg[ \sum_{k=2}^d (u_k^Tw_0)^2  \bigg].
\label{eq:Pwt-squared-final}
\end{align}
Plugging \eqref{eq:Pwt-squared-final} into \eqref{eq:trace-bound}, we finally obtain
\begin{align*}
E[w_{t}^TPM_kPw_{t}] \leq \eta^2 P_{1,t}(K) \sum_{k=2}^d E \big[ (u_k^Tw_0)^2 \big].
\end{align*}
\end{proof}

\begin{proof}[Proof of Lemma~\ref{lemma:single-epoch-convergence}]
From Lemma~\ref{lemma:uk-wt-square-exp-equality}, we have
\begin{align}
E[(u_1^T w_m)^2] = p_{m}(\alpha_1(\eta),\beta(\eta)) E[(u_1^Tw_0)^2] + 4\eta^2 \sum_{t=1}^{m-1} q_{m-t-1}(\alpha_1(\eta),\beta(\eta)) E[w_{t}^TPM_kPw_{t}]. \label{eq:u1-wm}
\end{align}
Using
\begin{align*}
E[w_{t}^TPM_kPw_{t}] = E[w_{t}^TP(C_t-C)u_ku_k^T(C_t-C)Pw_{t}] = E[(w_{t}^TP(C_t-C)u_k)^2] \geq 0
\end{align*}
and \eqref{eq:coefficient-positivity-general-alpha-q} in Lemma~\ref{lemma:coefficient-inequality} since $\alpha_1(\eta) > 4 \beta(\eta)$, we  have
\begin{align}
E[(u_1^T w_m)^2] \geq p_{m}(\alpha_1(\eta),\beta(\eta)) E[(u_1^Tw_0)^2].
\label{eq:u1-wm-bound}
\end{align}

On other hand, for $2 \leq k \leq d$, using Lemma~\ref{lemma:uk-wt-square-exp-equality} and \eqref{eq:coefficient-inequality} in Lemma~\ref{lemma:coefficient-inequality} since $\alpha_k(\eta) \leq \alpha_2(\eta) = 4\beta(\eta)$, we have
\begin{align*}
E[(u_k^T w_m)^2] &= p_{m}(\alpha_k(\eta),\beta(\eta)) E[(u_k^Tw_0)^2] + 4\eta^2 \sum_{t=1}^{m-1} q_{m-t-1}(\alpha_k(\eta),\beta(\eta)) E[w_{t}^TPM_kPw_{t}] \\
&\leq p_{m}(\alpha_2(\eta),\beta(\eta)) E[(u_k^Tw_0)^2] + 4 \eta^4 \sum_{t=1}^{m-1} q_{m-t-1}(\alpha_2(\eta),\beta(\eta)) E[w_{t}^TPM_kPw_{t}].
\end{align*}
Moreover, using Lemma~\ref{lemma:P-wt-square-inequality}, we further have
\begin{align}
E[(u_k^T w_m)^2] \leq p_{m}(\alpha_2(\eta),\beta(\eta)) E[(u_k^Tw_0)^2] + 4 \eta^4 \sum_{t=1}^{m-1} q_{m-t-1}(\alpha_2(\eta),\beta(\eta)) P_{1,t}(K) \sum_{k=2}^d E[ (u_k^Tw_0)^2].
\label{eq:uk-wm-bound}
\end{align}
Combining \eqref{eq:u1-wm-bound} and \eqref{eq:uk-wm-bound}, we obtain
\begin{align}
\frac{\sum_{k=2}^d E[ (u_k^Tw_m)^2]}{E[(u_1^Tw_m)^2]} 
&\leq 
\bigg( \frac{p_{m}(\alpha_2(\eta),\beta(\eta))}{p_{m}(\alpha_1(\eta),\beta(\eta))} + \frac{4 \eta^4 (d-1) \sum_{t=1}^{m-1} q_{m-t-1}(\alpha_2(\eta),\beta(\eta)) P_{1,t}(K)}{p_{m}(\alpha_1(\eta),\beta(\eta))} \bigg) \frac{\sum_{k=2}^d E[ (u_k^Tw_0)^2]}{E[(u_1^Tw_0)^2]}. \nonumber
\end{align}
Using \eqref{eq:coefficient-positivity-special-alpha} and \eqref{eq:coefficient-positivity-general-alpha} in Lemma~\ref{lemma:coefficient-inequality} since $\alpha_1(\eta) > 4\beta(\eta)$, we have
\begin{align*}
q_{m-t-1}(\alpha_2(\eta),\beta(\eta)) = (m-t)^2 (\beta(\eta))^{m-t-1}, \quad (\beta(\eta))^{m} = p_m(\alpha_2(\eta),\beta(\eta)) < p_m(\alpha_1(\eta),\beta(\eta)).
\end{align*}
Therefore, we obtain
\begin{align}
\frac{\sum_{k=2}^d E[ (u_k^Tw_m)^2]}{E[(u_1^Tw_m)^2]} 
\leq \bigg( \frac{p_{m}(\alpha_2(\eta),\beta(\eta))}{p_{m}(\alpha_1(\eta),\beta(\eta))} + \eta^4 P_{1,m}(K) \bigg) \frac{\sum_{k=2}^d E[ (u_k^Tw_0)^2]}{E[(u_1^Tw_0)^2]}
\label{eq:expectation-bound-raw-form}
\end{align}
where we used the fact that $\beta(n) > \epsilon$ for a fixed small enough $\epsilon$ and any $\eta \geq 0$.
\color{black}

Next, let
\begin{align*}
\rho(\eta,K) = g(\eta) + c^{\prime} \eta^4 K, \quad g(\eta) &= \frac{p_{m}(\alpha_2(\eta),\beta(\eta))}{p_{m}(\alpha_1(\eta),\beta(\eta))}.
\end{align*}
Using \eqref{eq:coefficient-positivity-special-alpha} and \eqref{eq:coefficient-positivity-general-alpha} in Lemma~\ref{lemma:coefficient-inequality}, we have
\begin{align}
g(\eta) &=  
\frac{4^{m+1} (\beta(\eta))^m}
{\big[ \big( \sqrt{\alpha_1(\eta)} + \sqrt{\alpha_1(\eta) - 4 \beta(\eta)} \big) ^{m} 
+ \big( \sqrt{\alpha_1(\eta)} - \sqrt{\alpha_1(\eta) - 4 \beta(\eta)} \big) ^{m}  \big]^2} \nonumber \\
&= 
\Bigg[
\frac{2^{m+1} (\sqrt{\beta(\eta)})^m}
{\big( \sqrt{\alpha_1(\eta)} + \sqrt{\alpha_1(\eta) - 4 \beta(\eta)} \big) ^{m} 
+ \big( \sqrt{\alpha_1(\eta)} - \sqrt{\alpha_1(\eta) - 4 \beta(\eta)} \big) ^{m}}
\Bigg]^2 \label{eq:g-eta} \\
&= 
\Bigg[
\frac{2^{m+1}}
{\big( \sqrt{\gamma(\eta)} + \sqrt{\gamma(\eta) - 4} \big) ^{m} 
+ \big( \sqrt{\gamma(\eta)} - \sqrt{\gamma(\eta) - 4} \big) ^{m}}
\Bigg]^2
\label{eq:g-eta-2}
\end{align}
where
\begin{align*}
\gamma(\eta) = \frac{\alpha_1(\eta)}{\beta(\eta)} = \frac{4(1-\eta+\eta \lambda_1)^2}{(1-\eta+\eta\lambda_2)^2}.
\end{align*}
By Lemma~\ref{lemma:g-eta}, we have $g(0)=1, g^{\prime}(0) = -2m^2(\lambda_1-\lambda_2)$, and $g^{\prime}(\eta) < 0$ for any $\eta \in (0,1]$, implying that $g(\eta)$ is a decreasing function of $\eta$ on $(0,1]$. Moreover, since $g(\eta)$ is twice continuously differentiable on an open interval containing $0$, by Taylor approximation at $\eta = 0$, we have
\begin{align}
g(\eta) = 1 - 2m^2(\lambda_1 - \lambda_2) \eta + o(\eta^{\delta})
\label{eq:main-term-small-oh}
\end{align}
for any $1< \delta<2$. For all subsequent analysis, any such $\delta$ would do it and we pick $\delta = 3/2$ arbitrarily. 

Plugging \eqref{eq:main-term-small-oh} into $\rho(\eta,K)$, we have
\begin{align*}
\rho(\eta, K) = 1- 2m^2(\lambda_1-\lambda_2)\eta + o(\eta^{3/2}) + c^{\prime} \eta^4 K.
\end{align*}
Since
\begin{align*}
\frac{\partial}{\partial \eta} \rho(\eta,K) \big|_{\eta = 0} < 0, \quad \rho(0,K)=1    
\end{align*}
hold, there exist some $\bar{\eta}(K)>0$ such that for every $0<\eta \leq \bar{\eta}(K)$, we have
\begin{align*}
\rho(1,0) \leq \rho(\eta,K)<1.     
\end{align*}
The lower bound follows the fact that $g(\eta)$ is decreasing and $c^{\prime} \eta^4 K \geq 0$.
\end{proof}

\begin{proof}[Proof of Theorem~\ref{theorem:algorithm-convergence}]
By Lemma~\ref{lemma:single-epoch-convergence}, there exists some $\bar{\eta}(K)$ such that for every $\eta \in (0,\bar{\eta}(K)]$, we have
\begin{align*}
0 <\rho(1,0) \leq \rho(\eta,K)<1.     
\end{align*}
By repeatedly applying
\begin{align*}
\frac{\sum_{k=2}^d E[(u_k^T\tilde{w}_{s})^2]}{E[(u_1^T\tilde{w}_{s})^2]}
=
\frac{\sum_{k=2}^d E[(u_k^Tw_m)^2]}{E[(u_1^Tw_m)^2]} 
\leq \rho(\eta,K) \frac{\sum_{k=2}^d E[(u_k^Tw_0)^2]}{E[(u_1^Tw_0)^2]} 
= \rho(\eta,K) \frac{\sum_{k=2}^d E[(u_k^T\tilde{w}_{s-1})^2]}{E[(u_1^T\tilde{w}_{s-1})^2]},
\end{align*}
we obtain
\begin{align*}
\frac{\sum_{k=2}^d E[(u_k^T \tilde{w}_s)^2]}{E[(u_1^T\tilde{w}_s)^2]}
\leq \rho(\eta,K)^s \frac{\sum_{k=2}^d E[(u_k^T\tilde{w}_0)^2]}{E[(u_1^T\tilde{w}_0)^2]} 
&= big \rho(\eta,K)^s \bigg( \frac{1-(u_1^T \tilde{w}_0)^2}{(u_1^T \tilde{w}_0)^2} \bigg).
\end{align*}
\end{proof}

\newpage
\subsection{Technical Lemmas}
\begin{lemma} \label{lemma:trace}
Let $w$ be a vector in $\mathbb{R}^d$, and let $P$, $M$ be $d \times d$ symmetric matrices. Then, we have
\begin{align*}
w^TPMPw \leq \| M \| \| Pw \|^2.
\end{align*}
\end{lemma}
\begin{proof}
By the cyclic property of the trace, we have
\begin{align*}
w^TPMPw = \text{Tr}[w^TPMPw] = \text{Tr}[MPww^TP].
\end{align*}
Since $Pww^TP$ is positive semi-definite, we have
\begin{align*}
\text{Tr}[MPww^TP] \leq \| M \| \text{Tr}[Pww^TP].
\end{align*}
Again, by the cyclic property of the trace, we finally have
\begin{align*}
w^TPMPw \leq \| M \| \text{Tr}[Pww^TP] = \| M \| \text{Tr}[w^TPPw] = \| M \| \| Pw \|^2.
\end{align*}
\end{proof}

\begin{lemma} \label{lemma:linear-algebra}
Let $w$ be a vector in $\mathbb{R}^d$ with $\|w\|=1$ and let $C$ be a $d \times d$ symmetric matrix. Then, for $P=I-ww^T$, we have
\begin{align*} 
\| PC - CP \|^2 = w^TC^2w - (w^TCw)^2.
\end{align*}
\end{lemma}
\begin{proof}
Let $U=PC-CP$. Since the non-zero singular values of $U$ are the square roots of the non-zero eigenvalues of $U^TU$, we focus on $U^TU$. By definition of $P$, we have
\begin{align*}
U = (I - ww^T)C - C(I-ww^T) = Cww^T - ww^TC,
\end{align*}
resulting in
\begin{align*}
U^TU &= (Cww^T - ww^TC)^T(Cww^T - ww^TC) \\
&= ww^TC^2ww^T - ww^TCww^TC - Cww^TCww^T + Cww^Tww^TC \\
&= (w^TC^2w)ww^T - (w^TCw)ww^TC - (w^TCw)Cww^T + Cww^TC.
\end{align*}
For any vector $u$ in $\mathbb{R}^d$, we have
\begin{align}
U^TUu &= (w^TC^2w)ww^Tu - (w^TCw)ww^TCu - (w^TCw)Cww^Tu + Cww^TCu \nonumber \\
&= \big[ (w^TC^2w)(w^Tu) - (w^TCw)(w^TCu) \big]w + \big[w^TCu - (w^TCw)(w^Tu) \big] Cw \label{eq:Ut-U-u}
\end{align}
meaning that that $U^TUu$ lies in the span of $w$ and $Cw$. This implies that any eigenvector $u$ for $U^TU$ corresponding to a non-zero eigenvalue is of the form 
\begin{align}
u = c_1w + c_2 Cw.
\label{eq:u}
\end{align}
By plugging \eqref{eq:u} into \eqref{eq:Ut-U-u}, we have
\begin{align*}
U^TUu &= c_1[ (w^TC^2w) - (w^TCw)^2 ]w + c_2[w^TC^2w - (w^TCw)^2] Cw \\
&= [ (w^TC^2w) - (w^TCw)^2 ] (c_1w + c_2Cw) \\
&= [ (w^TC^2w) - (w^TCw)^2 ] u.
\end{align*}
We conclude that all eigenvalues of $U^TU$ are $(w^TC^2w) - (w^TCw)^2$ and possibly $0$. Therefore, the spectral radius of $U^TU$ is $|(w^TC^2w) - (w^TCw)^2|$. Since it is easy to check that $P^2=P$ and the expansion of $\|PCw\|^2$ results in
\begin{align*}
\| PCw \|^2 = w^TCP^2Cw = w^TCPCw = (w^TC^2w) - (w^TCw)^2 \geq 0,
\end{align*}
we have
\begin{align*}
\| U \|^2 = \| PC - CP \|^2 = w^TC^2w - (w^TCw)^2.
\end{align*}
\end{proof}

\begin{lemma} \label{lemma:quadratic-form}
Let $C$ be a positive semi-definite $d \times d$ matrix and $(\lambda_1,u_1)$ be the largest eigenpair of $C$. Then, for any unit vector $w$ in $\mathbb{R}^d$, we have
\begin{align*} 
w^TC^2w - (w^TCw)^2 \leq 2 \lambda_1^2(1-(u_1^Tw)^2).
\end{align*}
\end{lemma}

\begin{proof}
Letting
\begin{align*}
w = (u_1^Tw) u_1 + (I - u_1u_1^T) w,
\end{align*}
we have after some manipulations
\begin{align}
w^TCw &= \lambda_1(u_1^T w)^2 + w^T(I-u_1u_1^T)C(I-u_1u_1^T)w
\label{eq:wt-C-w}
\end{align}
and
\begin{align}
w^TC^2w &= \lambda_1^2(u_1^T w)^2 + w^T(I-u_1u_1^T)C^2(I-u_1u_1^T)w.
\label{eq:wt-C-C-w}
\end{align}
Since the second terms in \eqref{eq:wt-C-w} and \eqref{eq:wt-C-C-w} are non-negative due to $C$ being positive semi-definite, we have
\begin{align*}
\lambda_1(u_1^Tw)^2 \leq w^TCw \leq \lambda_1, \quad \lambda_1^2(u_1^Tw)^2 \leq w^TC^2w \leq \lambda_1^2.
\end{align*}
Therefore,
\begin{align*}
w^TC^2w - (w^TCw)^2 \leq \lambda_1^2(1-(u_1^Tw)^4) = 2 \lambda_1^2(1+(u_1^Tw)^2)(1-(u_1^Tw)^2) \leq 2 \lambda_1^2( 1-(u_1^Tw)^2)
\end{align*}
where the last inequality follows from $(u_1^Tw)^2 \leq \|u_1\|^2 \|w\|^2 = 1$.
\end{proof}

\begin{lemma} \label{lemma:coefficient-inequality}
Let $w_t$ be a sequence of real numbers such that
\begin{align*}
w_t = (\alpha-\beta) w_{t-1} - \beta (\alpha-\beta) w_{t-2} + \beta^3 w_{t-3} + L_{t-1} + \beta L_{t-2}
\end{align*}
for $t \geq 3$ and $w_0 = L_0, w_1 = \frac{\alpha}{4} L_0, w_2 = \big( \frac{\alpha}{2}-\beta \big)^2L_0 + L_{1}$. Then, we have
\begin{align}
w_t = p_{t}(\alpha,\beta) L_0 + \sum_{r=1}^{t-1} q_{t-r-1}(\alpha,\beta) L_{r} \label{eq:coefficient-expression}
\end{align}
where $p_t(\alpha,\beta)$ and $q_t(\alpha,\beta)$ are recurrence polynomials defined as
\begin{align}
p_t(\alpha,\beta) &= (\alpha - \beta) p_{t-1}(\alpha,\beta) - \beta (\alpha - \beta) p_{t-2}(\alpha,\beta) + \beta^3 p_{t-3}(\alpha,\beta) \label{eq:coefficient-relation-p} \\
q_t(\alpha,\beta) &= (\alpha - \beta) q_{t-1}(\alpha,\beta) - \beta (\alpha - \beta) q_{t-2}(\alpha,\beta) + \beta^3 q_{t-3}(\alpha,\beta) \label{eq:coefficient-relation-q}
\end{align}
for $t \geq 3$ with
\begin{align}
p_0(\alpha,\beta) &= 1, \quad p_1(\alpha,\beta) = \frac{\alpha}{4}, \quad p_2(\alpha,\beta) = \Big( \frac{\alpha}{2} -\beta \Big)^2,
\label{eq:coefficient-initial-condition} \\
q_0(\alpha,\beta) &= 1, \quad q_1(\alpha,\beta) = {\alpha}, \quad q_2(\alpha,\beta) = (\alpha -\beta )^2.
\label{eq:coefficient-initial-condition-2}
\end{align}
Moreover, for $t \geq 0$, we have
\begin{itemize}
\item if $0 \leq \alpha = 4\beta$,
\begin{align}
p_t(4\beta,\beta) = \beta^t \geq 0, \quad q_t(4\beta,\beta) &= (t+1)^2 \beta^t \geq 0. \label{eq:coefficient-positivity-special-alpha}
\end{align}
\item if $0 \leq 4 \beta < \alpha$, \begin{align}
p_t(\alpha,\beta) &= \bigg[ \frac{1}{2} \bigg(\frac{\sqrt{\alpha}}{2} + \frac{\sqrt{\alpha - 4\beta}}{2} \bigg)^t + \frac{1}{2} \bigg(\frac{\sqrt{\alpha}}{2} - \frac{\sqrt{\alpha - 4\beta}}{2} \bigg)^t \bigg]^2 > p_t(4\beta,\beta) \geq 0, \label{eq:coefficient-positivity-general-alpha} \\
q_t(\alpha,\beta) &= \frac{1}{\alpha - 4 \beta} \bigg[ \bigg(\frac{\sqrt{\alpha}}{2} + \frac{\sqrt{\alpha - 4\beta}}{2} \bigg)^{t+1}  - \bigg(\frac{\sqrt{\alpha}}{2} - \frac{\sqrt{\alpha - 4\beta}}{2} \bigg)^{t+1} \bigg]^2 \geq 0. \label{eq:coefficient-positivity-general-alpha-q}
\end{align}
\item if $0 \leq \alpha < 4 \beta$, \begin{align}
p_t(\alpha,\beta) \leq p_t(4\beta,\beta), \quad q_t(\alpha,\beta) \leq q_t(4\beta,\beta). \label{eq:coefficient-inequality}
\end{align}
\end{itemize}
\end{lemma}

\begin{proof}
It is easy to check that $w_0$, $w_1$, and $w_2$ satisfy \eqref{eq:coefficient-expression}. Suppose that \eqref{eq:coefficient-expression} holds for $t-1,t-2,t-3$. Then, we have
\begin{align*}
w_t &= (\alpha-\beta) w_{t-1} - \beta (\alpha-\beta) w_{t-2} + \beta^3 w_{t-3} + L_{t-1} + \beta L_{t-2} \\
&= p_{t}(\alpha,\beta) L_0 + L_{t-1} + \alpha L_{t-2} + (\alpha-\beta)^2 L_{t-3} + \sum_{r=1}^{t-4} q_{t-r-1}(\alpha,\beta) L_r \\
&= p_{t}(\alpha,\beta) L_0  + \sum_{r=1}^{t-1} q_{t-r-1}(\alpha,\beta) L_{r}.
\end{align*}
Therefore, \eqref{eq:coefficient-expression} holds by induction.

Next, we prove \eqref{eq:coefficient-positivity-special-alpha}, \eqref{eq:coefficient-positivity-general-alpha}, \eqref{eq:coefficient-positivity-general-alpha-q} and \eqref{eq:coefficient-inequality}. The characteristic equation of \eqref{eq:coefficient-relation-p} is
\begin{align}
r^3 - (\alpha-\beta) r^2 + \beta (\alpha-\beta) r - \beta^3 = 0. \label{eq:characteristic-equation}
\end{align}

If $0 \leq \alpha = 4\beta$, \eqref{eq:characteristic-equation} has a cube root of $r=\beta$.
From initial conditions \eqref{eq:coefficient-initial-condition} and \eqref{eq:coefficient-initial-condition-2}, we obtain
\begin{align}
p_t(4\beta,\beta) = \beta^t \geq 0, \quad q_t(4\beta,\beta) = (t+1)^2 \beta^t \geq 0.
\label{eq:polynomial-closed-form-special-alpha}
\end{align}

If $0 \leq 4\beta < \alpha$, the roots of \eqref{eq:characteristic-equation} are 
\begin{align*}
r = \beta, \frac{\alpha-2\beta}{2} + \frac{\sqrt{\alpha^2-4\alpha\beta}}{2}, \frac{\alpha-2\beta}{2} - \frac{\sqrt{\alpha^2-4\alpha\beta}}{2}.
\end{align*}
With initial conditions \eqref{eq:coefficient-initial-condition}, we obtain
\begin{align*}
p_t(\alpha,\beta) &= \frac{1}{4} \bigg( \frac{\alpha-2\beta}{2} + \frac{\sqrt{\alpha^2-4\alpha\beta}}{2} \bigg)^{t}  + \frac{1}{4} \bigg( \frac{\alpha-2\beta}{2} - \frac{\sqrt{\alpha^2-4\alpha\beta}}{2} \bigg)^{t} + \frac{1}{2} \beta^{t} \\
&= 
\bigg[ \frac{1}{2} \bigg(\frac{\sqrt{\alpha}}{2} + \frac{\sqrt{\alpha - 4\beta}}{2} \bigg)^t + \frac{1}{2} \bigg(\frac{\sqrt{\alpha}}{2} - \frac{\sqrt{\alpha - 4\beta}}{2} \bigg)^t \bigg]^2.
\end{align*}
The second equality can be verified by expanding the square expression. 

By the Binomial Theorem and the fact that $\alpha > 4 \beta$, we have
\begin{align*}
p_t(\alpha,\beta) \geq \frac{1}{2} \bigg( \frac{\alpha - 2\beta}{2} \bigg)^{t} + \frac{1}{2} \beta^{t}  > \beta^{t} \geq 0.
\end{align*}
On the other hand, using \eqref{eq:coefficient-initial-condition-2}, we have
\begin{align*}
q_t(\alpha,\beta) & = \frac{1}{\alpha - 4 \beta} \bigg[ \bigg( \frac{\alpha-2\beta}{2} + \frac{\sqrt{\alpha^2-4\alpha\beta}}{2} \bigg)^{t+1}  + \bigg( \frac{\alpha-2\beta}{2} - \frac{\sqrt{\alpha^2-4\alpha\beta}}{2} \bigg)^{t+1} - 2 \beta^{t+1} \bigg] \\
& = 
\frac{1}{\alpha - 4 \beta} \bigg[ \bigg(\frac{\sqrt{\alpha}}{2} + \frac{\sqrt{\alpha - 4\beta}}{2} \bigg)^{t+1}  - \bigg(\frac{\sqrt{\alpha}}{2} - \frac{\sqrt{\alpha - 4\beta}}{2} \bigg)^{t+1} \bigg]^2 \\
& \geq 0.
\end{align*}
Again, the second equality can be established by expanding the square expression.

If $0 \leq \alpha < 4 \beta$, the roots of \eqref{eq:characteristic-equation} are 
\begin{align*}
r = \beta, \frac{\alpha-2\beta}{2} + \frac{\sqrt{4\alpha\beta-\alpha^2}}{2} i, \frac{\alpha-2\beta}{2} - \frac{\sqrt{4\alpha\beta-\alpha^2}}{2} i.
\end{align*}
Setting
\begin{align*}
\text{cos } \theta_p  = \frac{\alpha -2 \beta}{2\beta}, \quad \text{sin } \theta_p  = \frac{\sqrt{4\alpha\beta-\alpha^2}}{2\beta}
\end{align*}
it is easy to verify that
\begin{align*}
p_t(\alpha,\beta) &= \frac{1}{4} \beta^t \bigg[ \text{cos } \theta_p + i \text{ sin } \theta_p \bigg]^t + \frac{1}{4} \beta^t \bigg[ \text{cos } \theta_p - i \text{ sin } \theta_p \bigg]^t + \frac{1}{2} \beta^t \\
&= \frac{1}{4} (e^{i\theta t} + e^{-i\theta t}) \beta^t + \frac{1}{2} \beta^t \\
&= \frac{1}{4} | e^{i\theta t} + e^{-i\theta t} | \beta^t + \frac{1}{2} \beta^t \\
&\leq \frac{1}{4} (|e^{i\theta t}| + |e^{-i\theta t}|) \beta^t + \frac{1}{2} \beta^t \\
&= \beta^t.
\label{eq:polynomial-closed-form-general-alpha-trigonometric-p}
\end{align*}

Moreover, with
\begin{align*}
\text{cos } \theta_q  = \frac{\alpha -2 \beta}{2\beta}, \quad \text{sin } \theta_q  = \frac{\sqrt{4\alpha\beta-\alpha^2}}{2\beta}, \quad \text{cos } \phi_q  = 1-\frac{\alpha}{2\beta}, \quad \text{sin } \phi_q  = -\frac{\sqrt{4\alpha\beta-\alpha^2}}{2\beta},
\end{align*}
it can be seen by using elementary calculus that
\begin{equation}
q_t(\alpha,\beta) = \bigg[ \frac{2\beta}{4\beta-\alpha} + \frac{2\beta}{4\beta-\alpha} \text{cos} (\phi_q + t \theta_q) \bigg] \beta^t.
\label{eq:polynomial-closed-form-general-alpha-trigonometric-q1}
\end{equation}
Let
\begin{align*}
Q(t) = \frac{q_t(4\beta,\beta) - q_t(\alpha,\beta)}{\beta^t}.
\end{align*}
Then, from \eqref{eq:coefficient-relation-p} and \eqref{eq:coefficient-initial-condition}, we have
\begin{align}
Q(0) = 0, \quad Q(1) = \frac{4\beta - \alpha}{\beta}, \quad Q(2)= \frac{(4\beta-\alpha)(2\beta+\alpha)}{\beta^2}, \quad Q(3) = \frac{(\alpha^2 + 4\beta^2)(4\beta-\alpha)}{\beta^3}
\label{eq:polynomial-initial-condition-p-1}
\end{align}
resulting in
\begin{align}
Q(2)-Q(0) = \frac{(4\beta-\alpha)(2\beta+\alpha)}{\beta^2} \geq 0, \quad Q(3)-Q(1) = \frac{(\alpha^2 + 3\beta^2)(4\beta-\alpha)}{\beta^3} \geq 0. \label{eq:polynomial-initial-condition-p-2}
\end{align}
In order to show $Q(t) \geq 0$ for $t \geq 0$, we prove $Q(t+2)-Q(t) \geq 0$ for $t \geq 0$. Using \eqref{eq:polynomial-closed-form-special-alpha}, \eqref{eq:polynomial-closed-form-general-alpha-trigonometric-q1} and standard trigonometric equalities, it follows that
\begin{align*}
Q(t+2)-2Q(t)+Q(t-2) = 8 + \frac{2\alpha}{\beta} \text{cos}(\phi_q + t\theta_q).
\end{align*}
In turn, we have
\begin{align}
Q(t+2) - Q(t) &= Q(t) - Q(t-2) + 8 + \frac{2\alpha}{\beta} \text{cos} (\phi_q + t\theta_q) \nonumber \\
&\geq Q(t) - Q(t-2) + 8 - \frac{2\alpha}{\beta} \nonumber \\
&= Q(t) - Q(t-2) + \frac{2(4\beta-\alpha)}{\beta} \nonumber \\
&\geq Q(t) - Q(t-2). \label{eq:coefficient-inequality-relation-p}
\end{align}
From \eqref{eq:polynomial-initial-condition-p-1}, \eqref{eq:polynomial-initial-condition-p-2}, and \eqref{eq:coefficient-inequality-relation-p}, we obtain $Q(t) \geq 0$ for $t \geq 0$ implying
\begin{align*}
q_t(\alpha,\beta) \leq q_t(4\beta,\beta)
\end{align*}
for $t \geq 0$.
\end{proof}

\begin{lemma} \label{lemma:three-term-inequality-compact-form}
Let $w_t$ be a sequence of non-negative real numbers such that
\begin{align*}
w_t \leq \alpha w_{t-1} + \beta w_{t-2} + L_{t-1}
\end{align*}
for $t \geq 2$ with $w_0 = 0, w_1 \leq L_0$. If $\alpha, \beta \geq 0$, we have
\begin{align}
w_t \leq \sum_{l=0}^{t-1} r_{t-l-1}(\alpha,\beta) L_{l} \label{eq:three-term-inequality-compact-form}
\end{align}
where $r_t(\alpha,\beta)$ is a recurrence polynomial defined as
\begin{align}
r_t(\alpha,\beta) &= \alpha r_{t-1}(\alpha,\beta) + \beta r_{t-2}(\alpha,\beta) \label{eq:coefficient-relation-q}
\end{align}
for $t \geq 2$ with
\begin{align}
r_0(\alpha,\beta) = 1, \quad r_1(\alpha,\beta) = \alpha.
\label{eq:coefficient-initial-condition-q}
\end{align}
\end{lemma}

\begin{proof}
From $w_0 = 0$ and $w_1 \leq L_0$, it is obvious that \eqref{eq:three-term-inequality-compact-form} holds for $t = 0$ and $t=1$. Suppose that \eqref{eq:three-term-inequality-compact-form} holds for $t-1$ and $t-2$. Then, we have
\begin{align*}
w_t &\leq \alpha w_{t-1} + \beta w_{t-2} + L_{t-1} \\
&\leq \alpha \sum_{l=0}^{t-2} r_{t-l-2}(\alpha,\beta) L_{l} + \beta \sum_{l=0}^{t-3} r_{t-l-3}(\alpha,\beta) L_{l} + L_{t-1} \\
&= L_{t-1} + \alpha L_{t-2} + \sum_{l=0}^{t-3} (\alpha r_{t-l-2}(\alpha,\beta) + \beta r_{t-l-3}(\alpha,\beta)) L_{l} \\
&= \sum_{l=0}^{t-1} r_{t-l-1}(\alpha,\beta) L_{l}.
\end{align*}
Therefore, by mathematical induction, \eqref{eq:three-term-inequality-compact-form} holds for every $t$.
\end{proof}

\begin{lemma}
For
\begin{align*}
g(\eta) = \bigg[
\frac{2^{m+1}}
{h(\eta)}
\bigg]^2
\end{align*}
where
\begin{align*}
h(\eta) = \big( \sqrt{\gamma(\eta)} + \sqrt{\gamma(\eta) - 4} \big) ^{m} 
+ \big( \sqrt{\gamma(\eta)} - \sqrt{\gamma(\eta) - 4} \big) ^{m}, \quad 
\gamma(\eta) = \frac{4(1-\eta+\eta\lambda_1)^2}{(1-\eta+\eta\lambda_2)^2},
\end{align*}
we have 
\begin{align*}
g(0) = 1, \quad g^{\prime}(0) = -2m^2(\lambda_1-\lambda_2),  
\end{align*}
and
\begin{align*}
g^{\prime}(\eta) < 0    
\end{align*}
for any $\eta \in (0,1]$.
\label{lemma:g-eta}
\end{lemma}
\begin{proof}
Since $\gamma(0) = 4$, it is obvious that $g(0)=1$ holds. Next, by differentiating $\gamma(\eta)$, we have
\begin{align*}
\gamma^{\prime}(\eta) = \frac{8(1-\eta+\eta\lambda_1)(\lambda_1-\lambda_2)}{(1-\eta+\eta\lambda_2)^3} > 0
\end{align*}
for $\eta \in (0,1]$.

Using the chain rule on \eqref{eq:g-eta-2}, we obtain
\begin{align*}
g^{\prime}(\eta) &= - 2 \gamma^{\prime}(\eta) \bigg[ \frac{2^{m+1}}{h(\eta)} \bigg] \cdot \Bigg \{ \bigg[ \frac{1}{2\sqrt{\gamma(\eta)}} + \frac{1}{2\sqrt{\gamma(\eta)-4}} \bigg] \bigg[ \frac{m2^{m+1} \big( \sqrt{\gamma(\eta)}+\sqrt{\gamma(\eta)-4} \big)^{m-1}}{h(\eta)} \bigg]  \\
& \qquad + \bigg[ \frac{1}{2\sqrt{\gamma(\eta)}} - \frac{1}{2\sqrt{\gamma(\eta)-4}} \bigg] \bigg[ \frac{m2^{m+1} \big( \sqrt{\gamma(\eta)}-\sqrt{\gamma(\eta)-4} \big)^{m-1}}{ h(\eta)} \bigg]  \Bigg \} \\
&= - \frac{m 2^{m+1} \gamma^{\prime}(\eta) \sqrt{g(\eta)}}{h(\eta)^2}
\Bigg [
\frac{1}{\sqrt{\gamma(\eta)}} 
\Big[ {\big( \sqrt{\gamma(\eta)} + \sqrt{\gamma(\eta) - 4} \big) ^{m-1} 
+ \big( \sqrt{\gamma(\eta)} - \sqrt{\gamma(\eta) - 4} \big) ^{m-1}} \Big] \\
& \qquad + \frac{1}{\sqrt{\gamma(\eta)-4}} 
\Big[ 
{\big( \sqrt{\gamma(\eta)} + \sqrt{\gamma(\eta) - 4} \big) ^{m-1} 
- \big( \sqrt{\gamma(\eta)} - \sqrt{\gamma(\eta) - 4} \big) ^{m-1}} \Big] \Bigg].
\end{align*}
By the Binomial theorem, we have
\begin{align*}
\big( \sqrt{\gamma(\eta)} + \sqrt{\gamma(\eta) - 4} \big) ^{m-1} &=  \sum_{k=0}^{m-1} \binom{m-1}{k} \big( \sqrt{\gamma(\eta)-4} \big)^k \big( \sqrt{\gamma(\eta)} \big)^{m-k-1}
\end{align*}
and
\begin{align*}
\big( \sqrt{\gamma(\eta)} - \sqrt{\gamma(\eta) - 4} \big) ^{m-1} &=  \sum_{k=0}^{m-1} \binom{m-1}{k} \big( -\sqrt{\gamma(\eta)-4} \big)^k \big( \sqrt{\gamma(\eta)} \big)^{m-k-1},
\end{align*}
resulting in
\begin{align}
\big( \sqrt{\gamma(\eta)} + \sqrt{\gamma(\eta) - 4} \big) ^{m-1} 
+ \big( \sqrt{\gamma(\eta)} - \sqrt{\gamma(\eta) - 4} \big) ^{m-1}
= 2 \sum_{k=0}^{\left \lfloor{(m-1)/2}\right \rfloor} \binom{m-1}{2k} \big( \sqrt{\gamma(\eta)-4} \big)^{2k} \big( \sqrt{\gamma(\eta)} \big)^{m-2k-1}
\label{eq:appendix-binomial-theorem}
\end{align}
and
\begin{align*}
\big( \sqrt{\gamma(\eta)} + \sqrt{\gamma(\eta) - 4} \big) ^{m-1} 
- \big( \sqrt{\gamma(\eta)} - \sqrt{\gamma(\eta) - 4} \big) ^{m-1}
= 2 \sum_{k=0}^{\left \lfloor{({m-2})/{2}}\right \rfloor} \binom{m-1}{2k+1} \big( \sqrt{\gamma(\eta)-4} \big)^{2k+1} \big( \sqrt{\gamma(\eta)} \big)^{m-2k-2}.
\end{align*}
As a result, we have
\begin{align*}
g^{\prime}(\eta) &= - \frac{{m 2^{m+2} \gamma^{\prime}(\eta) \sqrt{g(\eta)}}}{h(\eta)^2} \cdot \Bigg[ {\sum_{k=0}^{\left \lfloor{(m-1)/2}\right \rfloor} \binom{m-1}{2k} \big( \sqrt{\gamma(\eta)-4} \big)^{2k} \big( \sqrt{\gamma(\eta)} \big)^{m-2k-2}} \\
& \qquad + {\sum_{k=0}^{\left \lfloor{(m-2)/2}\right \rfloor} \binom{m-1}{2k+1} \big( \sqrt{\gamma(\eta)-4} \big)^{2k} \big( \sqrt{\gamma(\eta)} \big)^{m-2k-2}} \Bigg].
\end{align*}
Since $\gamma^{\prime}(\eta) > 0$, $\gamma(\eta) > 4$ and $h(\eta)>0$ implying $\sqrt{g(\eta)}>0$ for $\eta \in (0,1]$, we have $g^{\prime}(\eta) < 0$ for any $\eta \in (0,1]$. Fact $h(\eta)>0$ can be established by using $m$ instead of $m-1$ in \eqref{eq:appendix-binomial-theorem}. Moreover, for $\eta=0$, we have
\begin{align*}
g^{\prime}(0) = -2m^2(\lambda_1 - \lambda_2).
\end{align*}
\end{proof}

\end{document}